\documentclass[12pt]{amsart}
\usepackage{enumitem}
\usepackage{xcolor}
\usepackage{amssymb}
\usepackage{bm}
\usepackage{comment}
\usepackage{caption}
\usepackage[a4paper, hmargin=2.6cm, vmargin={3.1cm, 3.1cm}]{geometry}
\usepackage{subcaption}
\usepackage{mathrsfs}
\usepackage[OT2,T1]{fontenc}
\usepackage{mathtools}
\mathtoolsset{showonlyrefs}
\usepackage{hyperref}
\hypersetup{colorlinks=true,
linktoc=all,linkcolor=blue,
urlcolor=black,anchorcolor=black,
citecolor=red}
\usepackage{setspace}
\numberwithin{equation}{section}
\numberwithin{figure}{section}

\newtheorem{thm}{Theorem}[section]

\newtheorem{lem}[thm]{Lemma}

\newtheorem{prop}[thm]{Proposition}

\theoremstyle{definition}
\newtheorem{defn}[thm]{Definition}

\theoremstyle{remark}
\newtheorem{rem}[thm]{Remark}

\newcommand{\say}[1]{``#1''}
\renewcommand{\C}{{\mathbb C}}
\newcommand{\D}{{\mathbb D}}
\newcommand{\T}{{\mathbb T}}
\newcommand{\R}{{\mathbb R}}

\newcommand{\Z}{{\mathbb Z}}
\newcommand{\N}{{\mathbb N}}

\newcommand{\calS}{\mathcal{S}}

\newcommand{\calH}{{\mathcal H}}
\newcommand{\calD}{{\mathcal D}}

\newcommand{\calB}{{\mathcal B}}

\newcommand{\calK}{{\mathcal K}}

\newcommand{\calU}{{\mathcal U}}

\newcommand{\fs}{\mathfrak{s}}

\newcommand{\diff}{{\mathrm d}}
\newcommand{\diffs}{\mathrm{ds}}
\newcommand{\m}{{m}}
\newcommand{\diffA}{\mathrm{d}\m}

\newcommand{\imag}{{i}}
\newcommand{\Ordo}{\mathrm{O}}
\newcommand{\ordo}{\mathrm{o}}

\newcommand{\e}{e}
\renewcommand{\Re}{\operatorname{Re}}
\renewcommand{\Im}{\operatorname{Im}}

\newcommand{\Biglog}{\mathbf{\mathrm{Log}}}
\newcommand{\BergP}{\Pi}
\newcommand{\hker}{K}
\newcommand{\Bhker}{B}
\newcommand{\BergK}{\calB}
\newcommand{\kerK}{\calK}
\newcommand{\kerPol}{\calS_\n}
\newcommand{\modkerK}{\widehat{\kerK}}
\newcommand{\holkerK}{\mathscr{K}}
\renewcommand{\ker}{K}
\newcommand{\n}{N}
\newcommand{\supp}{\mathrm{supp}\,}
\renewcommand{\epsilon}{\varepsilon}
\newcommand{\densityCurve}{v}
\newcommand{\densityBulk}{u}
\newcommand{\class}[1]{\{#1\}}
\makeatletter
\@namedef{subjclassname@2020}{
\textup{2020} \rm{MSC}}
\makeatother

\numberwithin{equation}{section}

\begin{document}

\title[Universality of outliers]
{Universality of outliers in weakly confined 
Coulomb-type systems}

\author[Butez]{Raphael Butez}
\address{\noindent Butez: Université de Lille \\
Département Mathématiques Cité Scientifique \\
59009 Villeneuve-d’Ascq \\ France}
\email{raphael.butez@univ-lille.fr}

\author[Garcia-Zelada]{David Garc{\'i}a-Zelada}
\address{Garc{\'i}a-Zelada: Laboratoire de Probabilités \\ 
Statistique et Modélisation\\ UMR
CNRS 8001\\ Sorbonne Université\\ 4 Place Jussieu\\ 75005 Paris\\ France}
\email{david.garcia-zelada@sorbonne-universite.fr}

\author[Nishry]{Alon Nishry}
\address{Nishry: School of Mathematical Sciences, 
Tel Aviv University, Tel Aviv, Israel}
\email{alonish@tauex.tau.ac.il}

\author[Wennman]{Aron Wennman}
\address{Wennman:
Department of Mathematics \\
KTH Royal Institute of Technology \\
S--100 44 Stockholm\\
Sweden} 
\email{aronw@kth.se}



\begin{abstract}
This work concerns \emph{weakly confined} 
particle systems in the plane, characterized by
a large number of outliers away from a 
\emph{droplet} where the bulk of 
the particles accumulate
in the many-particle limit. 
We are interested in the asymptotic behavior 
of outliers for two classes of point processes: 
Coulomb gases at determinantal inverse 
temperature confined by a regular background, 
and a class of random polynomials. 

We observe that 
the limiting outlier process
only depends on the \emph{shape} of the uncharged region
containing them, and the global net excess charge.
In particular, for a determinantal Coulomb gas confined 
by a sufficiently regular background measure, 
the outliers in a simply connected uncharged 
region converge to the corresponding Bergman point process.
For a finitely connected uncharged region 
$\Omega$, a family of limiting outlier processes arises, 
indexed by the (Pontryagin) dual of the fundamental group of $\Omega$. 
Moreover, the outliers in different uncharged regions
are asymptotically independent, even if the 
regions have common boundary points.
The latter result is a manifestation of screening properties of the 
particle system.
\end{abstract}

\maketitle

\section{Introduction and main results}
\label{s:intro}
\subsection{Weakly confined Coulomb gases}
\label{ss:weak-confine}
Let $\mu$ be a probability measure on the complex plane $\C$.
We are interested in
a system of $\n$ particles
of unit negative charge
attracted to the (positive)
charge distribution
$\kappa_\n \mu$,
where $\kappa_\n$ is such that
$\n< \kappa_\n\le \n+1$.
More precisely,
we consider the 
point process on $\mathbb C$
induced by the Boltzmann-Gibbs measure
\begin{equation}\label{eq:law}
\diff\mathbb{P}_{N}(z_1,\dots,z_\n)=\frac{1}{\mathcal Z_\n}
\e^{-2\calH_\n(z)}\diff \m_{\C^\n}(z),\qquad (z_1,\dots,z_\n)\in\C^\n
\end{equation}
where the Hamiltonian $\calH_\n$ is given by
\[
\calH_\n(z_1,\dots,z_\n)=-\sum_{i<j}\log|z_i-z_j| +
\kappa_\n\sum_{j}U^\mu(z_j),\qquad (z_1,\dots,z_\n)\in\C^\n,
\]
and where $\mathcal Z_\n$ is a normalizing constant. 
Throughout, $\m_{\C^\n}$ 
denotes the Lebesgue measure on $\C^\n$ and $U^\mu$ is
the logarithmic potential of the measure $\mu$
defined as
\[
U^\mu(z)=\int\log|z-w|\diff\mu(w),\qquad z\in\C.
\]
We denote by $\Theta_\n$ a random variable
with law
$\mathbb P_\n$. 
As the particles are interchangeable
and almost surely distinct, we may abuse 
notation and think of the random vector
as a simple point 
process $\Theta_{\n}=\{z_1,\ldots,z_\n\}$ 
(see Appendix~\ref{app:David}).

The point process
$\Theta_\n$ is sometimes called
a \emph{weakly confined
Coulomb gas} or a \emph{jellium} associated with $\mu$.
The idealized total background charge 
$\kappa_\n=\n$ would correspond to
global charge neutrality of the system, 
but we ask that $\kappa_\n>\n$ for
the law $\mathbb P_\n$ to be well-defined
without spatial truncation.
This choice is discussed by Jancovici in 
\cite[Section 2.1]{Jancovici84}
for $\kappa_\n=\n+1$ and also by Chafa{\"i}, Garc\'ia-Zelada and Jung in
\cite{ChafaiGarciaZeladaJung}.
The terminology
\emph{weakly confining}
is related to the growth of the potential
at infinity
$U^\mu(z)=\log|z| + \Ordo(1)$
(cf.\ Hardy \cite{Hardy}).
If the total background charge would instead be 
so large that $\liminf_\n\n^{-1}\kappa_\n>1$
and if $U^\mu$ is regular enough
then the gas would be strongly confined to a 
compact set $\mathcal{S}\subset\mathrm{supp}(\mu)$ 
(the \emph{droplet}), i.e.,
with high probability all particles 
would be found within a distance 
$\n^{-\frac12}\log\n$ from that set, see Ameur's
paper \cite{AmeurLocal}.

For both the weakly and strongly confined models, as 
the number $\n$ of particles tends to infinity,
the system arranges itself in equilibrium at
the macroscopic level. For the weakly confined gas this amounts to the 
convergence in law of the empirical measure
\begin{equation}
\label{eq:linear-stat-normalized}
\frac{1}{\n}\sum_{z\in\Theta_\n}\delta_{z} \longrightarrow \mu
\end{equation}
which reflects the fact that the particles are attracted
by the charge distribution $\mu$ and, to leading order, 
are able to neutralize it in the limit. 

The factor $2$ appearing in front of the Hamiltonian 
in \eqref{eq:law} plays 
a crucial role in our analysis. Indeed, 
the more general Coulomb gas obtained by replacing
$2$ by an \say{inverse temperature} $\beta>0$ is also well-studied, but
only for $\beta=2$ is 
the resulting point process 
a determinantal point process. 
Specifically, the Coulomb gas \eqref{eq:law} is a 
determinantal point process with correlation kernel
\begin{equation}\label{eq:kernel}
\mathcal{K}_\n(z,w)= \sum_{k=0}^{\n-1} P_{k,\n}(z)\overline{P_{k,\n}(w)} 
e^{-\kappa_\n U^\mu(z)} e^{-\kappa_\n \overline{U^\mu(w)}}, 
\quad z,w \in \mathbb{C}
\end{equation}
where $(P_{k,\n})_{k \leq \n-1}$ is any orthonormal basis of 
$L^2(\C, e^{-2\kappa_\n U^\mu}) \cap \C_{\n-1}[X]$,
and background measure given by Lebesgue measure 
on the plane, see Appendix~\ref{app:David}.

\begin{figure}[t!]
\centering
\includegraphics[width=.64\linewidth]{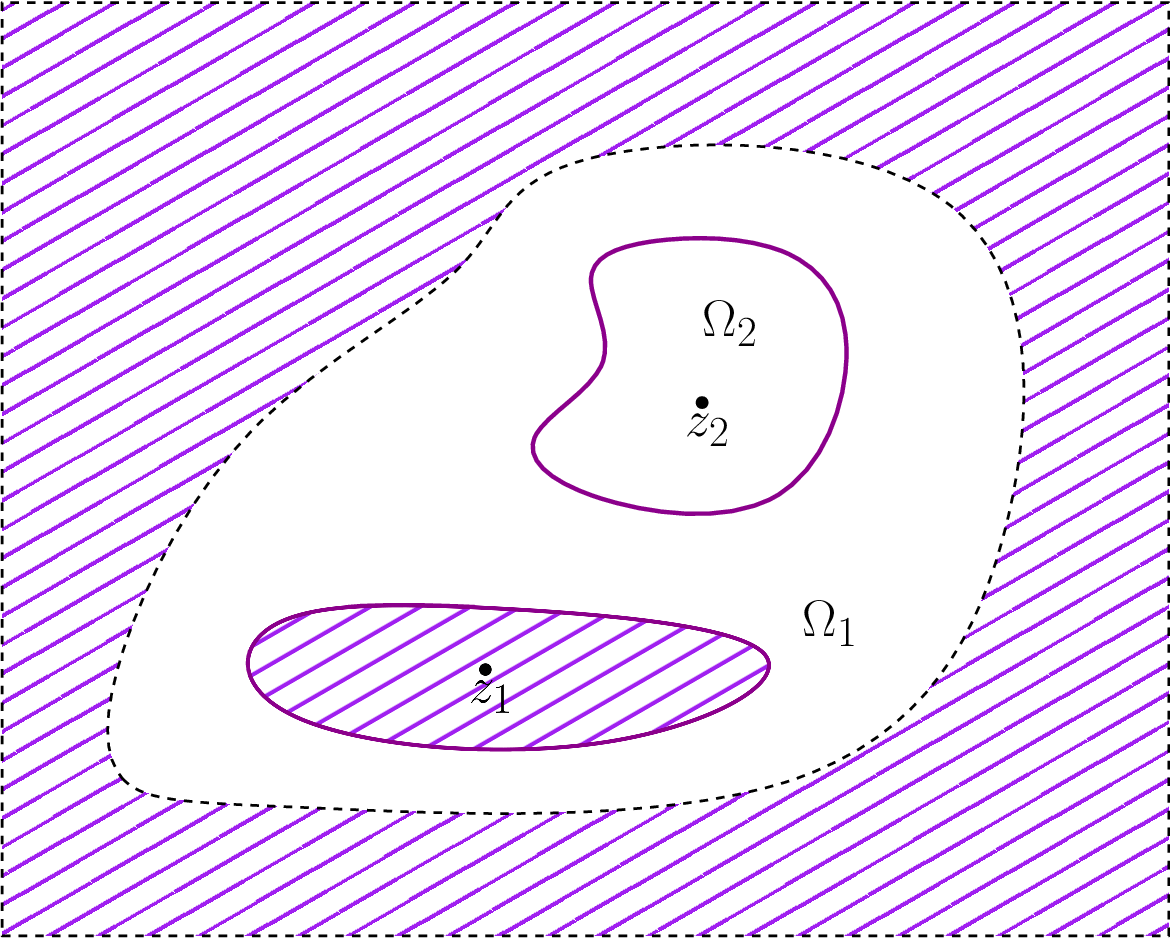}
\caption{Illustration of the support of admissible measures, with two
uncharged domains $\Omega_1$ and $\Omega_2$. 
For Theorem~\ref{thm:NonSimplyConnected}
and the the component $\Omega=\Omega_1$,
a possible choice of $z_1$ and $z_2$ is indicated.}
\label{fig:charge}
\end{figure}

\subsection{Bergman universality and independence.}\label{sub:Outliers}
Informally, we are interested in the
\emph{outlier point processes}
for the Coulomb gas $\Theta_\n$. 
Specifically, if $\Omega$ 
is an \emph{uncharged region}, i.e., 
a connected component of $\C\setminus\supp(\mu)$,
we want to determine
the asymptotic distribution of the point process
$\Theta_\n$
restricted to $\Omega$, namely
\begin{equation}
\label{eq:outlier}
\Phi_\n=\Theta_\n\cap \Omega,\qquad \n\ge 1.
\end{equation}
We say that $\Phi_\n$ converges to a point 
process $\Psi$ \emph{in law}, if for 
any $f\in C_c(\Omega)$, we have the convergence
\begin{equation}
\label{eq:outlier-convergence}
\sum_{z\in\Phi_\n} f(z)\to\sum_{z\in\Psi}f(z)
\end{equation}
in distribution. We call $\Psi$ the 
\emph{limiting macroscopic outlier process in $\Omega$}.
Notice that we do not normalize the linear statistics in 
\eqref{eq:outlier-convergence}. For the usual (strongly confined) 
Coulomb gases, there is no such limiting process.
However, in the setting of weak confinement, 
there tends to be a relatively
large number of particles at positive distance from $\supp(\mu)$.
In the recent work \cite{BG}, the first two authors
studied these systems under the assumption that the measure 
$\mu$ is radial, and it was found that
the outliers are governed by the so-called 
\emph{Bergman point process} of the uncharged region.

We will use methods from potential theory 
that require some regularity of $U^\mu$.
This is most sensitive near the boundary 
$\partial\Omega$, but to avoid technicalities
we require regularity also elsewhere. 
The main example we have in mind
is the
case where $\mathrm d \mu = \densityBulk\, \diff m_D$
for some nice enough density $\densityBulk>0$ and domain $D$
but we allow for charges
situated on curves as well.
Specifically, we require that 
the support of $\mu$
is a disjoint union
$\overline{D}\cup\Gamma$,
where $D$ is a domain with analytic boundary 
and where $\Gamma\cup\partial D$ is a finite disjoint 
union of closed analytic Jordan curves 
(see Figure~\ref{fig:charge}), $\Gamma\cap D=\emptyset$.

We use the standard notation $C^\infty$ and $C^\omega$ for the classes of
infinitely differentiable and real-analytic functions, 
respectively, and denote the arc length
measure on a disjoint union of curves $\gamma$ by $\sigma_\gamma$
and the restriction of the Lebesgue measure to $D$ by $m_D$. 

\begin{figure}[t!]
\centering
\includegraphics[width=.64\linewidth]{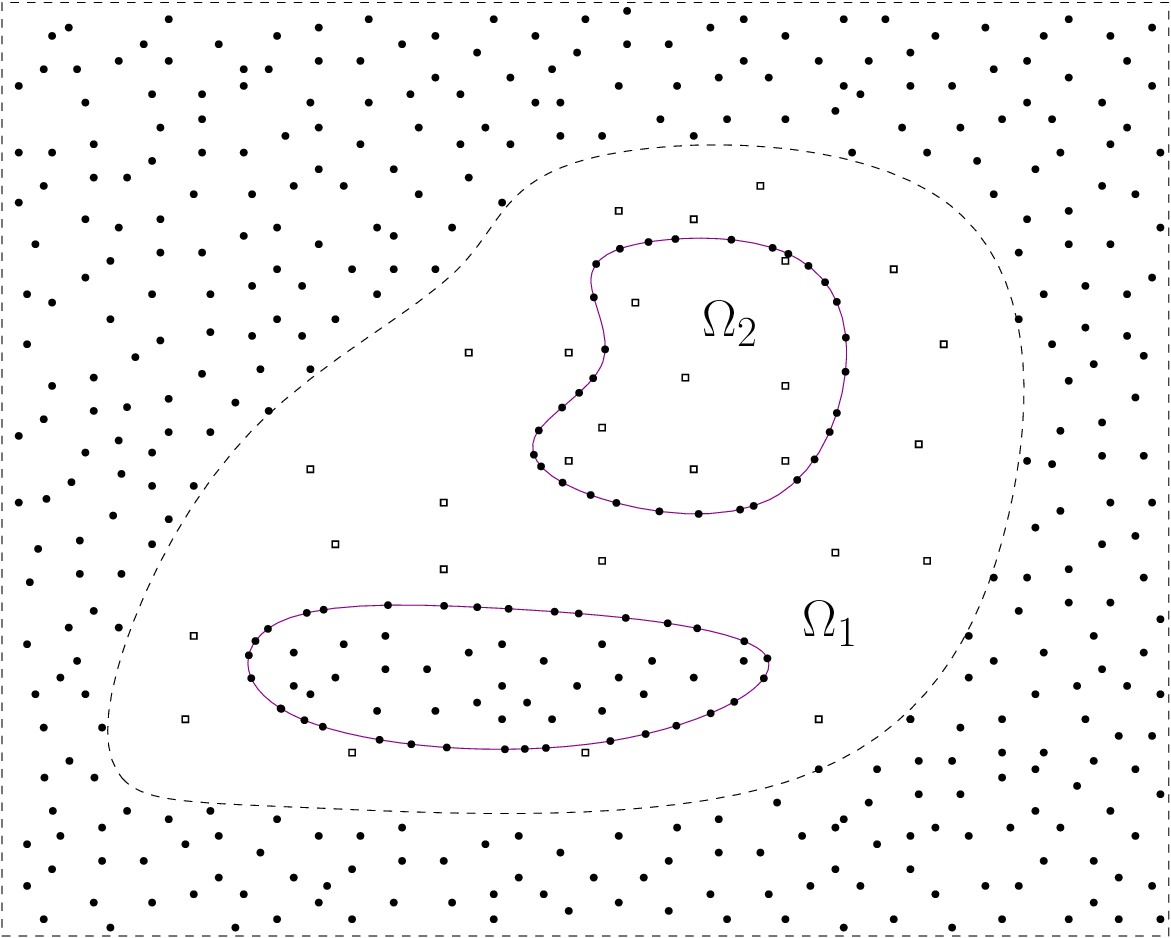}
\caption{Illustration of the the weakly confined jellium. 
The outliers are marked by white squares.}
\label{fig:points}
\end{figure}

\begin{defn}[Admissible probability measure]
\label{def:admissibility}
A probability measure $\mu$ is said to be \emph{admissible} 
if its support equals $\overline D\cup\Gamma$ 
for some $D$ and $\Gamma$ as above,
and
\[
\diff \mu=\densityBulk\,
\diff \m_D + \densityCurve  \,\diff \sigma_{\partial D\cup\Gamma},
\]
where $\densityBulk \in C^\infty(\overline{D})$
and $\densityCurve \in C^\omega(\partial D\cup\Gamma)$, 
$\densityBulk$ is real-analytic 
and strictly positive in a neighborhood of 
$\partial D$, and where $\densityCurve$ 
is either strictly positive or vanish 
identically on each component of
$\partial\Omega\cup\Gamma$.
\end{defn}

Notice that, since $\partial D\cup\Gamma$ 
is a finite union of closed curves in the plane,
either $\mathrm{supp}(\mu)$ or its complement
is bounded. By taking $D$ to be the empty set, we allow
the measure to be supported solely on a system of disjoint
Jordan curves.

We begin with a simple version of our main result, where
$\Omega$ is assumed to be a simply 
connected domain in the Riemann sphere
$\widehat{\C}=\C \cup \{\infty\}$.
For the formulation, 
we recall that the Bergman kernel of a domain $\Omega$
is the reproducing kernel
for the (closed) subspace of $L^2(\Omega,\diffA)$ consisting of
holomorphic functions.

\begin{thm}[Outliers in a simply connected region]
\label{thm:main-simple}
Suppose $\mu$ is an admissible probability
measure and
let $\Omega$ 
be a simply connected component of
$\widehat{\mathbb C}\setminus
\supp (\mu)$.
If
$\kappa_\n=\n+1$ we have the convergence
\[
\Phi_\n \xrightarrow[]{\phantom{aa}} \BergP_\Omega
\quad\text{in law}
\]
as $\n\to\infty$, where the limiting 
process $\BergP_\Omega$ in $\Omega$ is
the determinantal point process with correlation kernel
given by the Bergman kernel 
of $\Omega$.
\end{thm}

The convergence $\Phi_\n\to\BergP_\Omega$ amounts to locally uniform
convergence of the correlation kernels for the respective processes.
The kernel and the limiting process, called
the \emph{Bergman point process},
are described in Section~\ref{s:Bergman}. 
In the case where $\Omega$ is the unit disk, the Bergman point process 
coincides with the zeros of the hyperbolic GAF studied by Peres and Vir\'{a}g 
\cite{PeresVirag}. 
For a general simply connected domain $\Omega$, if $\phi$ is a conformal 
map from the unit disk $\mathbb{D}$ to $\Omega$, then, in law, 
$\BergP_{\Omega}= \phi(\BergP_{\mathbb D})$, which is an immediate 
consequence of the conformal equivariance of Bergman kernels.

\begin{rem}
The choice $\kappa_\n = \n+1$ is a natural from
a geometric point of view. The reason is that, with this choice, 
the Coulomb gas can naturally be viewed as a random process
on the sphere with respect to a continuous potential,
while any other choice of $\kappa_\n$ will add an extraneous charge at infinity.
\end{rem}

Theorem~\ref{thm:main-simple} is a particular case of the next theorem.
If $\Omega$ is multiply-connected, 
the outlier processes do not converge in general.
Instead, there is a whole family of possible limiting processes.
We identify these and characterize subsequential convergence.
\begin{thm}[Outliers in general regions]
\label{thm:NonSimplyConnected}
Suppose $\mu$ is an admissible probability
measure and
let $\Omega$ 
be a connected component of
$\mathbb C \setminus
\supp(\mu)$.
Let us fix $z_1,\dots,z_l$ in 
the interior of each of the holes of $\Omega$, i.e.,
in the bounded connected components
of $\mathbb C \setminus \overline{\Omega}$, and
let $q_1,\dots,q_l$ be the $\mu$-mass of each 
(closed) hole.
If, along some subsequence,
\begin{equation}\label{eq:conv-masses}
(e^{2\pi i \kappa_\n q_1},\dots,
e^{2\pi i \kappa_\n q_l})
\xrightarrow[]{\phantom{aaa}}
(e^{2\pi i Q_1}, \dots, e^{2\pi i Q_l})
\end{equation}
for some
$\mathbf Q = (Q_1,\dots,Q_l) \in \mathbb (\R/\Z)^l$,
then, along the same subsequence
\[
\Phi_\n \xrightarrow[]{\phantom{aa}}
\BergP_{\Omega,\mathbf Q} 
\mbox{ in law},
\]
where $\BergP_{{\mathbf Q},\Omega}$ is the weighted Bergman
point process on $\Omega$ associated with the weight
\begin{equation}\label{eq:weights}
\omega_{{\bf Q}}(z)=\prod_{i=1}^l |z-z_i|^{-2Q_i}.
\end{equation}
\end{thm}
The weighted Bergman kernel and its associated point 
process will be described in Section~\ref{s:Bergman}.
Theorem~\ref{thm:main-simple}
is a particular case of this result with $l=1$ and $\kappa_\n=\n+1$. 
Notice that by the uniqueness of the limit, the limiting
point process does not depend on the particular 
choice of the points $(z_i)_{i=1}^l$
(for a direct argument, see
Proposition~\ref{prop:Freedom}
and Proposition~\ref{thm: bergman continuity}).
Notice also that if ${\bf Q}\ne {\bf Q}'$ 
as elements of the torus $(\R/\Z)^l$, 
then the limiting processes are distinct 
(see Proposition~\ref{thm: bergman continuity}).

When there are several uncharged components, especially if these have common
boundary points, it is natural to ask if the limiting processes interact
with each other.
The following theorem answers this question.

\begin{thm}
[Independence between components]
\label{thm:indep}
Suppose $\mu$ is an admissible probability
measure.
Let $\Omega_1$ and $\Omega_2$ be any two connected
components of $\C\setminus 
\supp(\mu)$ and denote 
by $\Phi_{1,\n}$ and
$\Phi_{2,\n}$ the associated outlier processes. 
Along a subsequence for which \eqref{eq:conv-masses}
holds for both components, with some ${\bf Q}_1$
and ${\bf Q}_2$, 
we have that
\[
(\Phi_{1,\n}, \Phi_{2,\n})\xrightarrow[]{\phantom{law}}
(\BergP_{{\bf Q}_1,\Omega_1}, \BergP_{{\bf Q}_2,\Omega_2})
\quad\text{ in law},
\]
where 
$\BergP_{{\bf Q}_1,\Omega_1}$ and $\BergP_{{\bf Q}_2,\Omega_2}$
are independent.
\end{thm}

\subsection{Outliers for random zeros}
The Bergman universality of the outliers processes
goes beyond the weakly confined Coulomb gas model.
In support of this claim, we will show that the zero
set for one  model of random polynomials exhibits the same feature. 
We focus on this particular model for simplicity 
but the result should hold
in larger generality, for instance, 
for Weyl-type Gaussian polynomials
defined with respect to smooth exponentially varying weights.
In addition, the proof of 
Theorem~\ref{thm:random-pol} below
should be possible to adapt to the setting of
a finitely connected uncharged
region.
This is related to the paper by Katori and Shirai \cite{KatoriShirai},
which deals with zeros of random
Laurent series with i.i.d. coefficients.

The model we consider
belongs to a family of random polynomials
introduced by Zeitouni and Zelditch \cite{ZeitouniZelditch}.
Denote by $\nu$ a compactly supported probability measure on $\C$, and 
fix an orthonormal basis $(Q_{k,\n})_{k=0}^{\n}$ for the space
$\C_{\n}[z]$ of polynomials 
of degree at most $\n$, endowed with the inner product
of $L^2(\C,\e^{-2\n U^\nu}\diff\nu)$.
The random polynomial of degree $\n$ associated with $\nu$ is the 
linear combination of basis elements
\begin{equation}\label{eq:rand-pol}
P_\n=\sum_{k=0}^{\n}\xi_k Q_{k,\n},
\end{equation}
where $(\xi_k)_k$ is a sequence of i.i.d.\ standard complex
Gaussians, $\xi_k\sim\mathcal{N}_\C(0,1)$.

The law of $P_\n$ does not depend
on the specific orthonormal basis
$(Q_{k,\n})_{k=0}^\n$ we choose.
The zero set $\Psi_\n$ of $P_\n$ forms
a point process in the plane, which shares many features
with a weakly confined Coulomb gas whose background measure is $\nu$.
This was noticed in \cite{ZeitouniZelditch} and 
is discussed more thoroughly in \cite{BG}.
In particular, as $\n$ tends to infinity, the points of
$\Psi_\n$ distribute according to the measure 
$\nu$ and there is a large number of outliers.

We consider the class of measures
\[
\diff\nu=\densityCurve\,\diff\sigma_\Gamma
\]
where $\densityCurve$ is a strictly positive and real-analytic 
density 
with respect to the arc length measure $\sigma_\Gamma$ on
 a simple closed analytic curve $\Gamma$.
Let $\Omega_1$ and $\Omega_2$ denote the components of
$\hat \C\setminus \Gamma$, and $\Xi_{j,\n}=\Psi_{\n}\cap\Omega_j$ the point process of 
outliers in $\Omega_j$.

\begin{thm}\label{thm:random-pol}
As $\n\to\infty$, we have the convergence
\[
(\Xi_{1,\n}, \Xi_{2,\n}) \xrightarrow[]{\phantom{law}}
(\BergP_{\Omega_1}, \BergP_{\Omega_2}) \quad\text{ in law},
\]
where the Bergman point processes 
$\BergP_{\Omega_1}$ and
$\BergP_{\Omega_{2}}$ are independent.
\end{thm}

We believe that Bergman universality for outliers should hold 
also for more general measures $\nu$ which are 
not supported on a curve. However, 
we will not pursue this here.

\subsection{Description of the approach}

Our proofs are based on showing convergence of kernels: the correlation 
kernels of the determinantal point processes for the jellium, and the 
covariance kernels for the random polynomials. In the first case, 
locally uniform convergence towards the limiting Bergman kernel implies 
the convergence of the corresponding point process. In the second case, 
locally uniform convergence of the covariance kernels to the Szeg\H{o} 
kernel implies the convergence of the (analytic) Gaussian fields, 
which in turn implies the convergence of the zero processes. 
Here we discuss only the first case, and moreover assume for the 
sake of simplicity that ${\bf Q} = 0$.

For a description of the kernel $\kerK_\n(z,w)$ corresponding to 
the weakly confined jellium with $N$ particles, see 
Proposition~\ref{prop:CoulombAreDeterminantal}. However, we will work with a 
modification of it $\holkerK_\n(z,w)$ which is analytic in $(z, \bar{w})$. 
Our goal is to show the locally uniform convergence
\[
\holkerK_\n(z,w) \xrightarrow[N\to\infty]{} \BergK_\Omega(z,w),
\qquad z,w \in \Omega,
\]
where $\BergK_\Omega$ is the Bergman kernel of $\Omega$. 
In fact, the analytic nature of the kernels implies it is 
sufficient to prove convergence on the diagonal $w = z$.

An upper bound $\holkerK_\n(z,z) = \kerK_\n(z,z) \le \BergK_\Omega(z,z)$ 
on the diagonal follows from the extremal representation of 
the kernels (Lemma~\ref{lem:monotone}), and Montel's theorem 
reduces the task to proving the \emph{pointwise} lower bound
\begin{equation}\label{eq:liminf_bnd_kernels}
\liminf_{\n\to\infty} \kerK_\n(z,z) \ge \BergK_\Omega(z,z), 
\qquad z\in \Omega.
\end{equation}

We recall that
\[
\BergK_\Omega(z,z) = \sum_{k=1}^{\infty} \left| \psi_k(z)\right|^2, 
\qquad z\in\Omega,
\]
where $\left\{\psi_k\right\}_{k\ge 1}$ is (any) orthonormal 
basis for the Bergman space $A^2(\Omega)$, and similarly $\kerK_\n(z,z)$ 
can be expressed in terms of an appropriate sum of polynomials. 
In order to obtain the lower bound \eqref{eq:liminf_bnd_kernels}, 
we show that for a proper choice of $\psi_k$-s one can find, 
for every $k_0 \in \N$ fixed, an orthonormal set 
$\left\{P_{k,\n}\right\}_{k \le k_0}$ in 
$L^2_\n(\C,\e^{-2\kappa_\n U^\mu}\diff m_\C)$ 
consisting of polynomials in $\C_{\n-1}[z]$, such that,
\[
\left| P_{k,\n}(z) \right|^2 \e^{-2\kappa_\n U^\mu(z)} 
\xrightarrow[\n \to \infty]{} \left|\psi_k(z) \right|^2, 
\qquad z\in \Omega, \quad k \le k_0.
\]
Taking the limit $k_0 \to \infty$ we obtain 
\eqref{eq:liminf_bnd_kernels}. 
We emphasize these polynomials are not the 
standard orthogonal polynomials 
(ordered by degree), but a family adapted to 
the basis $\left\{\psi_k\right\}_{k\ge1}$.

The method we use to obtain the polynomials 
$P_{k,\n}$ is based on the approach developed in
by Hedenmalm and Wennman \cite{HW} (see also \cite{HW-ONP2} for a less 
technical exposition), which loosely
goes back to e.g.\ Berman's paper \cite{Berman} and the work \cite{AHM} 
of Ameur, Hedenmalm and Makarov, respectively, 
and even to H{\"o}rmander's paper 
\cite{HormActa65}. This uses a polynomial approximation
approach using H{\"o}rmander $\bar{\partial}$-estimates and
certain harmonic continuation properties
of the potential (Proposition~\ref{prop:extension1} and Lemmas
\ref{lem:horm} and
\ref{lem:changeofV}). These technical steps
are responsible for our regularity assumptions, but we believe
this is excessive.
We should stress that, due to the weak confinement, 
there are considerable simplifications in the 
construction of approximately orthonormal 
holomorphic functions when compared to \cite{HW}.
On the other hand, to allow us to consider bounded (cf.\ \cite{HW-OFF}) 
and multiply connected uncharged domains, 
we have to work with more flexible bases 
than the one given by standard orthogonal polynomials.

\subsection{Remarks and related work}
\subsubsection{Previous work}
If we specialize to the case when $\mu$ is 
the equilibrium measure of a compact set $K$,
then Theorem~\ref{thm:main-simple} is 
essentially contained in Sinclair
and Yattselev's work \cite{SY}. In this work, they also 
obtain detailed near-boundary blow-up asymptotics
for the correlation kernel. For the case of 
outliers in an annulus, one can obtain
Theorem~\ref{thm:NonSimplyConnected} 
by applying methods similar to 
\cite{BG,GarciaZelada}. 
Coulomb gases subject to a different scale of boundary confinement
is analyzed in the paper \cite{AmeurKangSeo} by Ameur, Kang and Seo.

Random polynomials similar to \eqref{eq:rand-pol} 
have appeared in the literature 
previously, see, e.g.,\ the work 
\cite{ShiffmanZelditch} of Shiffman-Zelditch. 
However, in their work the weight
function is not varying with $\n$, which simplifies matters. 
As a consequence, deriving Theorem~\ref{thm:random-pol} 
can be done more directly using 
general considerations on the Szeg\H{o} kernel, 
and in this context Shiffman and Zelditch 
obtained very precise results on the covariance 
kernels after a blowup. Note that in the special 
case when $\nu$ is the equilibrium
measure of a Jordan curve, the potential $U^\nu$ is constant, 
so this particular
instance becomes a special case of \cite{ShiffmanZelditch}.

\subsubsection{Some natural extensions}
For bounded uncharged regions $\Omega$, 
the assumption in Theorem~\ref{thm:NonSimplyConnected}
that $\mu$ is a probability measure is not necessary, 
and we make it only for convenience.
Indeed, for an external field $V$ of 
sufficient growth at infinity, we let $\mu_V$ 
denote the associated equilibrium measure $\mu_V$ in the sense of 
Saff and Totik \cite{SaffTotik}. Assume that $\Omega$
is a bounded finitely connected component of 
$\C\setminus\mathrm{supp}(\mu_V)$ such that
$\Delta V=2\pi\mu_V$ in a neighborhood of $\overline{\Omega}$, 
so that in particular $\Delta V=0$ in $\Omega$. Then, the proof of 
Theorem~\ref{thm:NonSimplyConnected} 
goes through word for word, provided that
$\mu_V$ satisfies an analogous regularity 
condition to Definition~\ref{def:admissibility}.

Finally, we would like to point out that
Theorem~\ref{thm:NonSimplyConnected}
could be generalized
to deal with line bundles over
Riemann surfaces as in the framework described by Berman
\cite{Berman}. In this geometric setting, 
the non-negativity condition of $\mu$
becomes a  non-negativity condition
on the curvature of the Hermitian line bundle and the
limiting point processes
obtained are similar to 
the ones described in Section~\ref{s:LogHarmonicBergman}. 

\subsubsection{Further questions}
\label{s:bound-number}
In this paper, we have only analyzed limiting 
outlier processes as $\n\to\infty$.
It is natural to also ask if it is possible to define 
outliers also for finite $\n$, but we do not
attempt to do this here. Other interesting 
questions that we do no pursue here concern the blow-up behavior of
$\Theta_\n$ near $\partial\Omega$, and possible limiting outlier
processes for weakly confined Coulomb gases at other inverse temperatures.
Finally, Theorem~\ref{thm:main-simple} and Theorem~\ref{thm:NonSimplyConnected} 
should hold under considerably less restrictive regularity conditions
on the measure $\mu$. We have made no attempt 
at optimizing our argument in this direction.

\subsection{Notation}
\label{s:not}
We write $\T, \D, \R$ and $\C$ for the unit circle, unit disk, 
real line and complex plane, respectively.
The extended complex plane
(or Riemann sphere) is denoted by $\widehat{\C}=\C\cup\{\infty\}$. 
We recall the definition
of the Wirtinger derivatives,
\[
\partial=\partial_z=\frac12\Big(\partial_x-\imag\partial_y\Big),\qquad
\bar\partial=\partial_{\bar z}
=\frac12\Big(\partial_x+\imag\partial_y\Big),\qquad z=x+\imag y.
\]
The Cauchy-Riemann equations are encoded in the operator $\bar\partial$
so that a complex differentiable function $f$
on a complex domain is holomorphic if $\bar\partial f=0$, and
its complex derivative is given by $f'(z)=\partial f(z)$.

We will make frequent use of standard 
notation to compare asymptotic quantities.
Specifically $f=\Ordo(g)$ and
$a\lesssim b$ mean that $f/g$ and $a/b$ are bounded
under the limit procedure in question. The notation $f=\ordo(g)$ means that
$f/g\to 0$, while $a\asymp b$ means that $a\lesssim b$ and $b\lesssim a$ holds
simultaneously. If, for two functions $f$ and $g$, 
the quotient $f/g$ is a non-zero 
constant, we write $f\propto g$.

We denote the {\em class of $x$} in $\mathbb R/\mathbb Z$ as
\[
\class{x}= x\, \mathrm{mod}\, 1\in \R/\Z.
\]
If we have a sequence $x_\n$ so that the classes 
converge $\class{x_\n}\to \class{x}$ 
with $x\in[0,1)$, then we choose to represent 
$\class{x_\n}$ by the unique real number
(also denoted $\class{x_\n}$, and depending on 
$x$)
with $x_\n-\class{x_\n} \in\Z$ and
\[
\class{x_\n} \in \big[x-\tfrac12,x+\tfrac12\big).
\]
Thus, if $\class{x_\n}$ is convergent as an element of the circle, 
there exists a sequence of integers $k_\n$ such that 
$x_\n=k_\n+\class{x_\n}$ and $\class{x_\n} \to x\in [0,1)$
as a real number.
In other words, we \emph{cut open} the circle $\R/\Z$ 
diametrically opposite to the limit point $x$.

We will use the notation $L^2(\Omega, \rho)$ for the $L^2$ space 
with respect to the measure $\rho\,\diffA_{\Omega}$ where $\m_{\Omega}$ 
is the restriction of the Lebesgue measures on $\C$ to $\Omega$. 

Finally, the number of points in a finite set $E$ is denoted by $\#(E)$.

\subsubsection{List of various kernels associated to the jellium}
For the convenience of the reader, we gather here brief definitions of 
the various kernels which will appear in 
our description of the jellium. Below, $\Omega$ is a domain and $\rho$
is a weight function on $\Omega$.

\begin{itemize}[leftmargin=5.5mm]
\item $K_{\rho, \Omega}$ is the Bergman kernel of the space 
$A^2_\rho(\Omega)$ of holomorphic 
functions on $\Omega$ in $L^2(\Omega,\rho\,\diffA)$. 
If $(\varphi_k)_{k \geq 1}$
is an orthonormal basis of $A^2_\rho(\Omega)$ then
\[
\hker_{\rho,\Omega}(z,w)
= \sum_{k=1}^\infty 
\varphi_k(z) \overline{\varphi_k(w)}, \quad  z,w \in \Omega.
\]
\item $\kerK_{\rho,\Omega}$ is the weighted Bergman kernel 
of the space $A^2_\rho(\Omega)$. It is defined by 
\[
\kerK_{\rho, \Omega} (z,w)= K_{\rho,\Omega}(z,w)
\rho(z)^{1/2}\rho(w)^{1/2}, \quad z,w \in \Omega.
\]
\item $B_{\Omega}$ is the classical Bergman kernel of $\Omega$.
It corresponds to the case where the weight $\rho$ is constant 
and equal to $1$. In this case, the weighted and unweighted Bergman 
kernels are equal: $B_{\Omega}=K_{1,\Omega}=\kerK_{1, \Omega}$. 
\item $\BergK_{\Omega, \mathbf Q}$ is a weighted Bergman kernel 
which extends the classical Bergman kernel $B_{\Omega}$ when $\Omega$ 
is not simply connected. This kernel is defined by the relation
\[
\BergK_{\Omega, \mathbf Q} =\mathcal K_{\omega_{\mathbf Q}, \Omega}.\] 
Here, $\bf Q$ and $	\omega_{{\bf Q}}$ are defined in \eqref{eq:weights}.
\item $K_{\n}$ is the polynomial Bergman kernel associated to a determinantal 
point process. It is only defined in the appendix and is not used in the 
body of the paper. 
\item $\kerK_{\n}$ is the correlation kernel of the determinantal point process 
$\Theta_\n$ given by \eqref{eq:law} (with Lebesgue measure
as background measure; see Appendix~\ref{app:David}).
By definition, if we set $\rho_{\n}=e^{-\kappa_\n U^\mu}$, 
then $\kerK_{\n} = \kerK_{\rho_\n,\C}$.
\item $\modkerK_\n$ is a modified version of $\kerK_{\n}$, which is only 
defined on some $\Omega' \supset \Omega$ and which coincides with 
$\kerK_{\n}$ on $\Omega$. This kernel is only an intermediate step 
in order to define the holomorphic kernel is the next item.
\item $\holkerK_{\n}$ is the holomorphic kernel and is defined 
by the relation 
\[
\holkerK_{\n}(z,w)=\modkerK_\n(z,w)\omega_{\class{\kappa_\n{\bf q}}}^{-\frac12}(z)
\omega_{\class{\kappa_\n{\bf q}}}^{-\frac12}(w), \quad z,w \in \Omega.
\]
As its name suggests, it is holomorphic and defines the 
same process as $\kerK_{\n}\big\vert_{\Omega^2}$ on $\Omega$.
\end{itemize}

\section{Background on Bergman spaces and potential theory}
In this section we have collected preliminary
material on Bergman spaces and potential theory.
We begin in Section~\ref{s:Bergman} by
recalling the standard notion
of Bergman spaces and Bergman kernels
in the weighted setting. Then,
in Section~\ref{s:LogHarmonicBergman},
we consider
log-harmonic weights and define
the notion of a Bergman point process that appear in
our theorems. Section~\ref{ss:harm-ext}
is about an harmonic extension of the potential while
Section~\ref{TransformationKerne}
deals with finding an equivalent kernel.
\subsection{Bergman spaces and kernels}
\label{s:Bergman}

Let $\mathcal D \subset \mathbb C$ 
be an open set and let $\rho$ be a continuous 
positive {\em weight function} on
$\mathcal D$. The Bergman space $A^2_\rho(\mathcal D)$, also denoted
by $A^2(\mathcal D,\rho)$, is
the collection of all holomorphic functions
in $L^2(\mathcal D,\rho)$. It is well-known
that $A^2_\rho(\mathcal D)$ is a reproducing kernel
Hilbert space, and we denote its reproducing kernel by
$\hker_{\rho,\mathcal D}(z,w)$.
More precisely, $\hker_{\rho,\mathcal D}$ is the
kernel of
the orthogonal projection of  
$L^2(\mathcal D,\rho)$ onto $A^2_\rho(\mathcal D)$
and if $(\varphi_k)_{k \geq 1}$
is an orthonormal basis of $A^2_\rho(\mathcal D)$ then
\[
\hker_{\rho,\mathcal D}(z,w)
= \sum_{k=1}^\infty 
\varphi_k(z) \overline{\varphi_k(w)},
\]
where the convergence in the series is locally
uniform.
We denote by $\kerK_{\rho,\mathcal D}$ the weighted
\emph{correlation kernel}, which is given by
\[
\kerK_{\rho,\mathcal D}(z,w)=K_{\rho,\mathcal D}(z,w)
\rho^{\frac12}(z)\rho^{\frac12}(w).
\]
When the weight $\rho$ involves a parameter $\n$, we will
write $K_\n$ and $\kerK_\n$ for the kernels,
and $\lVert \,\cdot\,\rVert_\n$ and $\langle \cdot,\cdot\rangle_\n$
for the $L^2$-norm and inner product, 
respectively, when no confusion should occur.

For weakly confining weights $\e^{-2\kappa_\n U^\mu}$, the 
Bergman space $A^2(\C,\e^{-2\kappa_\n U^\mu})$ 
coincides with the $\n$-dimensional space $\C_{\n-1}[z]$ of polynomials
of degree at most $\n-1$ (see \eqref{eq:holom-infinity} below).  
If $\mathcal D$ is not the whole space, then
there may of course be non-polynomial elements 
of $A^2(\mathcal D,\e^{-2\kappa_\n U^\mu})$,
and we instead introduce the polynomial subspaces 
\[
A^2_\n(\mathcal D,\e^{-2\kappa_\n U^\mu})
:=A^2(\mathcal D,\e^{-2\kappa_\n U^\mu})\cap\C_{\n-1}[z].
\]

\subsection{Bergman point processes on an uncharged background}
\label{s:LogHarmonicBergman}
Let $\mathcal D \subset \mathbb C$ be an open set
and fix a positive weight function $\rho$ on $\mathcal{D}$.

\begin{defn}
\label{def:Bergman-process}
The \emph{weighted Bergman point process} $\BergP_{\rho,\mathcal{D}}$
associated with the weight $\rho$ is
the determinantal point process on $\mathcal D$
associated with the kernel 
$\mathcal K_{\rho,\mathcal D}$ and
the Lebesgue measure restricted
to $\mathcal D$. 
\end{defn}

When the weight is constant $\rho\equiv 1$, we write simply 
$\BergP_{\mathcal{D}}$ for the
\emph{(unweighted) Bergman point process}.
For weights $\omega_{\mathbf{Q}}$ of the form 
\eqref{eq:weights} on a domain $\Omega$, 
we denote the weighted Bergman point process by
$\BergP_{\mathbf{Q},\Omega}$. As the notation 
suggests, $\BergP_{\mathbf{Q},\Omega}$
only depends on the class $\mathbf{Q}$ of 
the weight $\omega_{\mathbf{Q}}$.
The remainder of this section is devoted to 
analyzing such processes and the
associated kernels.

\medskip

The following proposition
shows that the law of $\BergP_\rho=\BergP_{\rho,\calD}$ is stable under a
wide class of perturbations of the weight.

\begin{prop}
\label{prop:Freedom}
Let $\mathcal D \subset \mathbb C$ 
be an open set and let $\rho$ be a continuous strictly positive
function on $\mathcal D$. 
Fix a zero-free holomorphic function $f$
on $\mathcal D$ and let
$\rho_{f} = |f|^2\rho$. Then,
\[
\BergP_{\rho} = 
\BergP_{\rho_f}\quad \text{ in law}.
\]
In particular, 
$\mathcal K_{\rho}(z,z) =
\mathcal K_{\rho_f}(z,z)$ for every
$z \in \mathcal D$. Moreover, if $\rho$
and $\widetilde \rho$ are two strictly positive 
continuous weights on $\mathcal D$ such that
$\BergP_{\rho} = 
\BergP_{\widetilde \rho}$ in law and if we suppose that
$\BergP_{\rho}$ is not the empty point process, then
\[\widetilde \rho=|f|^2\rho\]
for some zero-free holomorphic function
$f$ on $\mathcal D$.
\end{prop}

\begin{proof}
Let us first prove the first part.
Suppose that $\rho$ is continuous and strictly positive
and that $f$ is zero-free and holomorphic.
By the definition of determinantal point processes,
we need to show that
$\det\left(\mathcal K_{\rho}(x_i,x_j)_{1 \leq i,j \leq k}\right)
=\det\left(\mathcal K_{\rho_f}(x_i,x_j)_{1 \leq i,j \leq k}\right)$ 
for every $k \geq 1$ and any points
$x_1,\dots,x_k \in \mathcal D$.
But, since $\rho_f= |f|^2 \rho$,
the map given by
\[g \in L^2(\mathcal D,\rho_f ) \mapsto
fg \in L^2(\mathcal D,\rho) \] 
is a well-defined isometry 
that sends
$A^2_{\rho_f}(\mathcal D)$ onto 
$A^2_{\rho}(\mathcal D)$
so that the corresponding orthogonal projections 
are conjugates of each other,
which is written as
\[K_{\rho}(z,w) =
f(z) K_{\rho_f}(z,w) f(w)^{-1}\]
for every $z,w \in \mathcal D$. By the definition of
$\mathcal K$, this implies that
\[\mathcal K_{\rho}(z,w) =
\frac{f(z)}{|f(z)|} \mathcal K_{\rho_f}(z,w) 
\frac{\overline{f(w)}}{|f(w)|},\]
for every $z,w \in \mathcal D$, which 
implies that the laws of
$\BergP_{\rho}$ and $\BergP_{\rho_f}$ are the same
and that 
$\mathcal K_{\rho}(z,z) =
\mathcal K_{\rho_f}(z,z)$ for every $z \in \mathcal D$.

For the second part, let
$\rho$ and $\widetilde \rho$ be two
continuous and strictly positive weights on $\mathcal D$
such that $\BergP_{\rho} = 
\BergP_{\rho_f}$ in law.
For simplification, we use the notation 
$K=K_{\rho}$ and 
$\widetilde K = K_{\widetilde \rho}$ for the kernels, which 
we recall are 
holomorphic in the first argument
and anti-holomorphic in the second. 
We also use the notation
$\mathcal K = \mathcal K_{\rho}$
and $\widetilde {\mathcal K} = 
{\mathcal K}_{\widetilde \rho}$ for the weighted (correlation) kernels.
The equality in law $\BergP_{\rho} = 
\BergP_{\widetilde \rho}$ implies that
\[
\det\left(\mathcal K(z_i,z_j)_{1\leq i,j\leq k}\right) = 
\det\left(\widetilde{\mathcal K}(z_i,z_j)_{1 \leq i,j \leq k}\right).
\]
for every $k \geq 1$ and 
$z_1,\dots,z_k \in \mathcal D$.
In particular,  we have that
$\mathcal K(z,z) = \widetilde{\mathcal K}(z,z) $
while
$\mathcal K(z,z)\mathcal K(w,w) - 
\mathcal K(z,w)\mathcal K(w,z) = 
\widetilde{\mathcal K}(z,z) \widetilde{\mathcal K}
(w,w) - 
\widetilde{\mathcal K}(z,w)
\widetilde{\mathcal K}(w,z)  $
for every $z,w \in \mathcal D$, which are the $k =1$
and $k=2$ cases.
Since the kernels are Hermitian, these two
cases imply
that $|\mathcal K(z,w)|
=|\widetilde{\mathcal K}(z,w)|$
for every $z,w \in \mathcal D$. Then,
\begin{equation} 
\label{eq:EqualityOfNorms}
|K(z,w)| \rho(z)\rho(w)
=|\widetilde{K}(z,w)|
\widetilde{\rho}(z)\widetilde{\rho}(w)
\end{equation}
for every $z,w \in \mathcal D$. Now, either
$K$ and $\widetilde K$
are zero everywhere
or there is some
$w_0 \in \mathcal D$ such that
the maps
$z \mapsto K(z,w_0)$ and
$z \mapsto \widetilde K(z,w_0)$
are non-zero holomorphic functions whose
zero sets coincide (and are discrete).
The first case is discarded  because $\BergP_{\rho}$
is not the empty process so that we can
consider the meromorphic function $f$ defined by
$f(z) = K(z,w_0)\rho(w_0)/(\widetilde{K}(z,w_0)\widetilde{\rho}(w_0))$.
Then, \eqref{eq:EqualityOfNorms} implies that
\[
|f(z)|
= \frac{\widetilde{\rho}(z)}{\rho(z)}
\]
for every $z \in \mathcal D$ such that 
$\widetilde K(z,w_0) \neq 0$.
Since $\widetilde{\rho}(z)/\rho(z)$ is locally bounded, $f$
has only removable singularities
so that it can be extended holomorphically to $\mathcal D$
and, by continuity, $|f|= \widetilde{\rho}/\rho$ on $\mathcal D$.
Taking logarithms completes the proof.
\end{proof}

Write $\rho = e^{-2V}$ for
some continuous function
$V$ on $\mathcal D$. We say that $\rho$
is log-harmonic if $V$ is harmonic, i.e., $\Delta V = 0$.
In view of Gauss's law of electrostatics, 
$V$ can be thought of as a possible potential
for an uncharged domain $\mathcal D$ which makes
it plausible that the
weights $\rho=e^{-2V}$ are the relevant ones 
in the context of outliers.

\begin{rem}[Flat line bundles] \label{rem:flat}
The notion of a log-harmonic measure $\rho$
makes sense on any open Riemann surface $M$
since any holomorphic change of coordinates
preserves the log-harmonicity.
In this way, Definition~\ref{def:Bergman-process}
can be used to define an analogous family of
Bergman point processes on $M$
and
Proposition~\ref{prop:Freedom} also holds in this case.
Moreover,
the freedom described in
Proposition~\ref{prop:Freedom}
is closely related 
to the classification of flat Hermitian line bundles
on $M$. This is not a coincidence
since, by noticing that
every line bundle is trivial on the open Riemann surface $M$, 
we may rephrase
our setting by considering Hermitian line bundles
on $M$  similarly to \cite{Berman} and the log-harmonic measures
would essentially play the role of
flat Hermitian metrics.
In particular, the set of equivalence
classes obtained by 
making two log-harmonic measures 
$\mu$ and $\widetilde \mu$ equivalent if
and only if 
$\diff \mu = |f|^2 \diff \widetilde \mu$ for some 
zero-free holomorphic function $f$ on $M$
is in one-to-one correspondence with the set
of morphisms from the fundamental group $\pi_1(M)$ 
to the unit circle $U(1)$ 
(the Pontryagin dual of $\pi_1(M)$).
This is stated
for a particular case 
in Proposition~\ref{thm: bergman continuity}.
\end{rem}

Below, we will consider domains appearing as uncharged
regions, and weights of the following form. If $\Omega$ is 
an $l$-connected domain,
we let
$z_1,\dots,z_l \in \mathbb C \setminus \Omega$ be a collection of points, 
one in each hole of $\Omega$, and let
$V(z)=
{\bf Q}\cdot \Biglog(z)$ for some
$\mathbf Q \in \mathbb R^l$,
where we have defined
\[
\Biglog(z) =
\Biglog_{z_1,\dots,z_l}(z):=
\Big(\log |z-z_1|, \dots, \log|z-z_l|\Big).
\]
This gives us a
weight of the form \eqref{eq:weights},
\[
\omega_{\bf Q}(z)=\e^{-2{\bf Q}\cdot \Biglog(z)}.
\]
Suppose, for simplicity,
that the holes of $\Omega$
have non-empty interior. Then,
by Axler \cite[p. 248]{Axler} and
the previous proposition, it follows that 
for every harmonic function $V$
there exists a unique
$\mathbf Q \in [0,1)^l$ such that
$\BergP_{e^{-2V}} = \BergP_{e^{-2\mathbf Q \cdot \Biglog}}$.
Moreover, for general
$\mathbf Q =(Q_1,\dots,Q_l) \in \mathbb R^l$, the Bergman process
$\BergP_{\mathbf Q \cdot \Biglog(z)}$ only depends
on $(e^{2\pi i Q_1},\dots,e^{2\pi i Q_l})$
or, equivalently, on 
$(\class{Q_1},\dots,\class{Q_l}) \in 
(\mathbb R \setminus \mathbb Z)^l$.
For this reason, there is no harm
in thinking about $\mathbf Q$ as an element of 
$(\mathbb R \setminus \mathbb Z)^l$.
To keep the notation short, we will denote the Bergman process 
$\BergP_{e^{-2\mathbf Q \cdot \Biglog(z)}}$
by $\BergP_{\mathbf Q}:=\BergP_{e^{-2\mathbf{Q}\cdot \Biglog}}$ 
and furthermore the weighted kernel
$\mathcal K_{\omega_{\mathbf Q}, \Omega}$
by $\BergK_{\Omega, \mathbf Q}$.
We notice that although the kernel
$\BergK_{\Omega, \mathbf Q}$ depends
on $\mathbf Q \in \mathbb R^l$,
the diagonal restriction
$\BergK_{\Omega, \mathbf Q} (z,z)$
depends only on the class
$(\class{Q_1},\dots,\class{Q_l})$.

\begin{rem}[Conformal equivariance of weighted Bergman processes]
\label{conformal_equivariance_bergman}
Suppose that $\phi$ is a conformal mapping of a domain
$\Omega_{1}$ onto another domain $\Omega_2$.  
Then, for any weight $\omega$ on $\Omega_1$, 
we have $\phi(\Pi_{\omega,\Omega_1}) 
= \Pi_{\omega \circ \phi, \Omega_2}$.
In particular, if $\Omega_1$ is $l$-connected, then 
for any $\mathbf Q  \in \mathbb (\R/\Z)^l$, there exists 
$\mathbf Q'  \in \mathbb (\R/\Z)^l$ 
(directly determined by $\bf Q$ and $\phi$)
such that $\phi(\BergP_{\Omega_1,\mathbf Q}) 
= \BergP_{\Omega_2,\mathbf Q'} $ in law.	
\end{rem}

Let $\mathcal C_\Omega$ denote the space
of point configurations on $\Omega$ 
(see Section~\ref{app:David}) and let
$\mathcal P(\mathcal C_\Omega)$ be
the set of probability measures 
on $\mathcal C_{\Omega}$ endowed with
the weak topology. 

\begin{prop}[Continuous injective
parametrization]
\label{thm: bergman continuity}
The function defined by\break
 $(z,\mathbf Q) \in \Omega \times (\mathbb R/\mathbb Z)^l
\mapsto \BergK_{\Omega,{\bf Q}}(z,z)$ 
is continuous. Moreover,
the map  from
$(\mathbb R/\mathbb Z)^l$
to 
$\mathcal P(\mathcal C_\Omega)$
that associates
$\mathbf Q 
\in (\mathbb R/\mathbb Z)^l$
to the law of $\BergP_{\bf Q}$
is injective and continuous.

\end{prop}
\begin{proof}
The injectivity 
is a consequence of
Proposition~\ref{prop:Freedom}
together with the fact that
$\prod_{i}|z-z_i|^{Q_i}$ is not the modulus
of a holomorphic function if 
$Q_i \in (0,1)$ for some $i \in \{1,\dots,l\}$.
The continuity claims
are a consequence of Pasternak's theorem
\cite[Theorem~5.1]{Pasternak}.
\end{proof}

\subsection{An upper bound on one-point functions}

Next, we state a lemma
indicating a relation between Bergman kernels 
with log-harmonic weights
and the kernel for the outlier process. 
We use the notations of 
Section~\ref{s:intro}. 
Take an
$l$-connected component
$\Omega$
of $\mathbb C \setminus \supp (\mu)$, where $\mu$ 
is admissible, and 
points $z_1,\dots,z_l \in 
\mathbb C \setminus \Omega$
in each of the holes of $\Omega$.
Denote
the $\mu$-mass of the $i$-th hole of $\Omega$
by $q_i$ and put ${\bf q}=(q_i)_i$.

\begin{lem}[Diagonal domination]
\label{lem:monotone}For every $N \geq 1$ 
and $z \in \Omega$
\[
\kerK_\n(z,z)
\leq \BergK_{\Omega, {\kappa_{\vspace{1mm}\n}} {\bf q}} (z,z).
\]
\end{lem}

\begin{proof} We start by noticing that, for weakly confining 
weights $\omega_\n=\e^{-2\kappa_\n U^\mu}$, the 
Bergman space $A^2(\C,\omega_\n)$ 
coincides with the $\n$-dimensional space $\C_{\n-1}[z]$ of polynomials
of degree at most $\n-1$ (see \eqref{eq:holom-infinity} below 
combined with Liouville's theorem), hence  $\kerK_\n = \kerK_{\omega_\n,\C}$.
In addition, we notice that
$\BergK_{\Omega,\kappa_\n {\bf q}}(z,z)=
\kerK_{{\omega_N}|_\Omega, \Omega}(z,z)$. This is a consequence
of Lemma~\ref{lem: holom extension}
together with
Proposition~\ref{prop:Freedom}.
If we denote by
$\|f\|_N$ the norm of 
$f$ in $L^2(\mathbb C,\omega_N)$
and by
$\lVert f\rVert_{\n,\Omega}$
the norm of 
$f|_{\Omega}$ in
$L^2(\Omega,\omega_N|_\Omega)$,
we can write the extremal representation of the kernels
\[
\kerK_\n(z,z)=\sup_{f\in A^2_{\omega_\n}(\C)}\frac{|f(z)|^2
\omega_\n(z)}{\lVert f\rVert^2_{\n}}
\quad \text{ and }
\quad
\kerK_{{\omega_N}|_\Omega, \Omega}(z,z)
= \hspace{-3mm}\sup_{f\in 
A^2(\Omega,\omega_\n|_\Omega)}
\hspace{-2mm}
\frac{|f(z)|^2
\omega_\n(z)}{\lVert f\rVert^2_{\n,\Omega}}.
\]
Then, since
$\lVert f\rVert_\n\ge \lVert f\rVert_{\n,\Omega}$ and
$A^2_{\omega_\n}(\C)
\subset A^2(\Omega,\omega_\n|_\Omega)$,
we find the sought inequality.
\end{proof}

\subsection{Harmonic extension of potentials.}
\label{ss:harm-ext} 
Assume that $f$ is a holomorphic function on a domain $\Omega$, and that 
for some $\epsilon>0$ and $z_0\in\partial\Omega$, 
$I:=\partial\Omega\cap\D(z_0,\epsilon)$ 
is a real-analytic arc. By Carathéodory's theorem, 
$f$ extends continuously to $I$. Suppose further that
$f$ maps $\Omega\cap\D(z_0,\epsilon)$ 
univalently onto a domain in the unit disk, 
such that $f(I)$ is a sub-arc of the circle. 
Then $f$ extends to a univalent function of a neighborhood
of $I$ (see, for instance, Pommerenke's book \cite[Theorem 6.2]{Pommerenke}).
In fact, similar extension properties hold for 
any holomorphic or harmonic function which is
defined on one side of an analytic curve with
real-analytic extension to the boundary.
We will routinely apply such extensions,
and will usually not make any distinction between
the original function and the extended one in terms of notation.

The only function we will deal with for which the extension property
is not entirely obvious is the potential $U^\mu$.
\begin{prop}[Harmonic extension]
\label{prop:extension1}
Under the assumptions of Theorem~\ref{thm:NonSimplyConnected},
the harmonic function $U^\mu\vert_{\Omega}$ extends
past $\partial\Omega$ to a harmonic function $\calU=\calU^\mu$ 
on an open set $\Omega'$ containing $\overline{\Omega}$.
\end{prop}

This result should be absolutely standard, but since we did not locate
an appropriate reference we supply a proof in 
Appendix~\ref{app:proof-harm-ext}.

\begin{prop}
\label{prop:extension2}
Let $\Gamma$ be a boundary component of 
$\partial\Omega$ and let $z_0\in\Gamma$.
Denote by ${\rm n}_0$ the outward unit
normal to $\Omega$ at $z_0$.
Under the assumptions of Theorem~\ref{thm:NonSimplyConnected},
we have
\begin{equation}\label{eq:bound-below}
\big(U^\mu-\mathcal{U}^\mu\big)(z_0+t{\rm n}_{z_0})=
\begin{cases}
c_1t + \Ordo(t^2),& \mu(\Gamma)>0 \\
c_2 t^2 + \Ordo(t^3),& \mu(\Gamma)=0,
\end{cases}
\end{equation}
for $0\le t\le t_0$, where $c_i>0$.
In addition, in each case the implicit constant 
is uniform for $z_0\in\partial\Omega$, and $c_i=c_i(z_0)$ 
depends smoothly on $z_0$.
\end{prop}

\begin{proof}
Assume without loss of generality that $z_0=0$ and that
the normal points in the imaginary direction. We denote the relevant
connected component of $\partial\Omega$ by $\Gamma$.
We split $U^\mu$ as $U^{\mu_1}+U^{\mu_2}$ where 
$\mu_1$ has a density with respect
to $1_{\Omega}dm$ locally near $0$, and where $\mu_2$ is supported on
$\Gamma$. Similarly, we let
$\mathcal{U}=\mathcal{U}_1+\mathcal{U}_2$, where $\mathcal{U}_i$ are harmonic
in $\Omega$ and $U^{\mu_i}-\mathcal{U}_i=0$ on $\Omega$ (that this splitting
of the harmonic extensions is possible follows from the proof of
Proposition~\ref{prop:extension1}). We perform a Taylor expansion
at $\Gamma$ of each contribution $U^{\mu_i}-\mathcal{U}_i$ separately. 
We first find that $(U^{\mu_1}-\mathcal{U}_1)$ vanishes on $\Omega$ 
and is $C^{1,1}$--smooth,
so its gradient also vanishes along $\Gamma$. 
Since $\Gamma$ is tangent to the real line at $0$, a Taylor expansion 
shows that $\partial_x^2(U^{\mu_1}-\mathcal{U}_1)(0)$ equals the
second order tangential derivative of $U^{\mu_1}-\mathcal{U}_1$ at $0$,
which implies that $\partial_x^2(U^{\mu_1}-\mathcal{U}_1)(0)=0$.
This observation allows us to rewrite the second 
derivative
\[
\partial_{y}^2 (U^{\mu_1}-\mathcal{U}_1)(0)=
\partial_{x}^2 (U^{\mu_1}-\mathcal{U}_1)(0)+\partial_{y}^2 
(U^{\mu_1}-\mathcal{U}_1)(0)=\Delta U^{\mu_1}(0),
\]
and hence, we obtain
\[
(U^{\mu_1}-\mathcal{U}_1)(\imag y)=\frac12\partial^2_y 
(U^{\mu_1}-\mathcal{U}_1)(0) y^2 + \Ordo(y^3)
=\frac12\Delta U^{\mu_1}(0)y^2 + \Ordo(y^3),\qquad y>0.
\]
If $\densityCurve\equiv 0$ on $\Gamma$, we are done. 
If not, we have $\densityCurve(0)>0$,
in which case we have
\begin{equation}
\label{eq:taylor-U21}
(U^{\mu_2}-\mathcal{U}_2)(x+\imag y)=(2\pi)\densityCurve(0)y
+\Ordo(y^2),\qquad y>0.
\end{equation}
Hence, we have established the desired formula also in this case.

By compactness of $\Gamma$, the error terms are uniform provided
that $U^{\mu_i}-\mathcal{U}_i\vert_{\Omega^c}$ are $C^3$-smooth in a
one-sided neighborhood of $\Gamma$. But in fact
these functions extend real-analytically to a full 
neighborhood of $\Gamma$, so the uniformity claim follows. That $c_i$
depend smoothly on $z_0\in\Gamma$ is likewise clear.
\end{proof}

\subsection{A transformation of the kernel}
\label{TransformationKerne}
We recall that
for an $l$-connected component
$\Omega$
of $\mathbb C \setminus \supp (\mu)$, where $\mu$ 
is admissible, we fixed
points $z_1,\dots,z_l \in 
\mathbb C \setminus \Omega$
in each of the holes of $\Omega$.
The mass of the $i$-th hole of $\Omega$
is denoted by $q_i$, and we put ${\bf q}=(q_i)$.
We let $\Omega'\supset\overline{\Omega}$ 
be as in Section~\ref{ss:harm-ext}, so that the
harmonic extension $\calU^\mu$ of $U^\mu\vert_{\Omega}$
is harmonic on the domain $\Omega'$. 

The following result is a direct consequence 
of the \emph{logarithmic
conjugation theorem} from Axler
\cite[p. 248]{Axler}.

\begin{lem} \label{lem: holom extension}
There exists a holomorphic function $\mathfrak h$ on $\Omega'$ such that
\[
\calU^\mu(z) - \mathbf q \cdot \Biglog(z)
= \Re \mathfrak h(z),
\]
so that we have
\begin{equation}
\e^{-\kappa_\n \calU^{\mu}(z)}
= \e^{-\class{\kappa_\n {\bf q}}\cdot \Biglog(z)}\, \Big|
\e^{-\kappa_\n \mathfrak h(z)} \prod_{i=1}^l
(z-z_i)^{-(\kappa_\n q_i-\class{\kappa_\n q_i})}\Big|.
\end{equation}
\end{lem}

We introduce the \emph{transformed correlation kernel}
\begin{equation}
\label{eq:transformed-kernel}
\modkerK_\n(z,w)=
\e^{-\imag \kappa_\n\Im(\mathfrak h(z))}
\kerK_\n(z,w)
\e^{\imag \kappa_\n\Im(\mathfrak h(w))},\qquad (z,w)\in\Omega'.
\end{equation}
In the appendix we recall some basic properties of determinantal processes,
including the fact \eqref{eq:ChangeOfK}
that
if $c:\Omega\to R$ is a continuous function, then
the kernels $\kerK_N(z,w)$ and
\[
\modkerK_\n(z,w):=\e^{\imag c(z)}\kerK_N(z,w)
\e^{-\imag c(w)},
\]
determine the same determinantal point process.
Hence, the restrictions of the kernels $\modkerK_\n$ and
$\kerK_{\n}$ to $\Omega \times \Omega$
 induce the same process.

We obtain a holomorphic kernel after division
by $(\omega_{\class{\kappa_\n {\bf q}}}(z)
\omega_{\class{\kappa_\n {\bf q}}}(w))^{1/2}$, which
is locally uniformly bounded in $\n$. So, we put
\begin{equation}
\label{eq:holKer}
\holkerK_{\n}(z,w)=
\modkerK_\n(z,w)\omega_{\class{\kappa_\n{\bf q}}}^{-\frac12}(z)
\omega_{\class{\kappa_\n{\bf q}}}^{-\frac12}(w).
\end{equation}
This allows us to use a normal family argument
to deduce precompactness of the family of transformed kernels
whenever $(\class{\kappa_\n {\bf q}})_{\n\in \N}$ converges.

\section{Convergence of weighted polynomial bases}
\label{s:basis}
\subsection{Construction of basis functions}
Fix an admissible measure $\mu$ and let
$\Omega$ be an $l$-connected uncharged
domain
let $q_1,\dots,q_l$ be the $\mu$-mass of each hole in
$\Omega$ and let $z_1,\dots,z_l$ be fixed points in each of the holes.
We recall from Section~\ref{s:not} that 
the number $\class{\kappa_\n q_j}\in\R/\Z$ 
is represented by the unique real number 
(also denoted $\class{\kappa_\n q_j}$)
such that $\kappa_\n q_j-\class{\kappa_\n q_j}\in\Z$ and 
$\class{\kappa_\n q_j}\in [q-\frac12,q+\frac12)$, 
where $q$ is a representative of the limit.
To simplify matters, we abuse notation 
and write that $(\class{\kappa_\n {\bf q}})_{\n \in \N}$
converges towards ${\bf Q} \in [0,1)^l$,
which may only be true along a subsequence $\mathcal{L}$.

The point process $\Phi_\n$ is the determinantal process
associated with the correlation kernel $\kerK_\n$
described in Section~\ref{s:Bergman}. The outlier process
is given by the restricted kernel $\kerK_\n\vert_{\Omega\times\Omega}$,
which determines the same process as the restriction of the modified
kernel $\modkerK_\n(z,w)$.
To prove the convergence of the point process
$\Psi_\n$ towards the Bergman point process
$\BergK_{\Omega,{\bf Q}}$, due to Proposition~\ref{prop:UCVofKernels}, it is sufficient to
show that the kernels $\modkerK_\n$ converge
uniformly on any compact set of $\Omega \times \Omega$ towards
$\BergK_{\Omega,{\bf Q}}$. This is equivalent to asking that the modified
holomorphic kernels $\holkerK_\n$ defined in \eqref{eq:holKer}
converge locally uniformly
towards the holomorphic Bergman kernel $\Bhker_{\Omega,{\bf Q}}$.
The next normal families lemma takes advantage of the
holomorphicity of $\holkerK_\n$ to obtain precompactness.
Once that is established,
we only need to show that there is only one possible limit point.
\begin{lem}
[Normal families]
\label{lem: normal families}
The sequence $(\holkerK_\n)_{\n \in \N}$ of holomorphic
kernels is precompact in the
compact-open topology.
\end{lem}

\begin{proof}

The kernels in the sequence  $(\holkerK_\n)_{\n \in \N}$ 
are analytic in $(z,\bar w)$. 
By Montel's theorem for holomorphic functions in several variables 
(see e.g\ \cite[Ch.~1, Proposition~6]{Narasimhan}), it is 
enough to show that $(\holkerK_\n)_{\n \in \N}$ 
is uniformly bounded on any compact 
subset of $\Omega \times \Omega$. By the reproducing kernel property and 
the Cauchy-Schwarz inequality, for any $(z,w) \in \Omega \times \Omega$,
\[
|\holkerK_\n(z,w)|^2 \leq \holkerK_\n(z,z) \holkerK_\n(w,w). 
\]
So, it is sufficient to prove that the 
kernels are uniformly bounded on the diagonal. 
Combining Lemma~\ref{lem:monotone} and Proposition~\ref{thm: bergman continuity} 
gives, on any compact set $K$
\[
\holkerK_\n(z,z) \leq \BergK_{\Omega,\class{\kappa_\n {\bf q}}}(z,z) 
\leq \sup_{(z,\mathbf Q)\in 
(\mathbb R/\mathbb Z)^l \times K} \BergK_{\Omega,\mathbf Q}(z,z) < \infty.
\]
\end{proof}

To obtain the main result, it is enough to show that
if $\class{\kappa_\n {\bf q}}$ converges to ${\bf Q}$, 
then we have convergence of holomorphic kernels along the diagonals
\[
\forall z \in \Omega , \quad \lim_{\n \to \infty}
\holkerK_{\n}(z,z) =
\Bhker_{\Omega, {\bf Q}}(z,z).
\]
Indeed, if $\holkerK$ is the limit of $\holkerK_\n$:
\[
\holkerK(z,w)=\lim_{\n\to\infty}\holkerK_\n(z,w).
\]
then $\holkerK$ coincides with the kernel
$\Bhker_{\Omega, {\bf Q}}$ on the diagonal of $\Omega \times \Omega$,
and the next lemma shows that this is enough to complete the proof.
\begin{lem}[Lemma 2.5.1 from \cite{HoughKrisPeresVirag}, p.30]
\label{lem:Herimitan}
Let $L(z,w)$ be analytic in $z$ and anti-analytic in $w$
(i.e. analytic in $\bar{w}$) for $(z,w) \in \Omega \times \Omega$.
If $L(z,z)=0$ $ \forall z \in \Omega$,
then $L(z,w)=0$ $\forall z,w \in \Omega$.
\end{lem}

In the remaining part of the proof, we show the 
pointwise convergence of the diagonal restrictions of
the holomorphic kernels $\holkerK_\n$ towards $\Bhker_{\Omega,\bf Q}$.
In the next lemma, we construct a sequence of
orthonormal bases in which we will express
the weighted Bergman kernels $\BergK_{\Omega, \class{\kappa_\n {\bf q}}}$.

\begin{prop}\label{prop:good-basis}
There exists an orthonormal basis 
$(\psi_k)_k$ of $A^2(\Omega, \e^{-2{\bf Q}\cdot\Biglog(z)})$, such that 
moreover $\psi_k\in A^2(\Omega, \e^{-2{\mathbf{Q}_\n}\cdot\Biglog(z)})$.
If $(\psi_{k,\n})_{k \in \N}$ is the orthonormal set obtained
from $(\psi_k)_k$ by 
the Gram-Schmidt procedure in
$A^2(\Omega, \e^{-2\mathbf{Q}_\n\cdot\Biglog(z)})$, we have
\begin{enumerate}[label=(\roman*)]
\item For any $z \in \Omega$, $\psi_{k,\n}(z)
\xrightarrow[\n \to \infty]{} \psi_k(z)$;
\item All the functions
$\psi_{k,\n}$ and $\psi_k$ for $k\le k_0$ extend holomorphically to 
some domain $\Omega'_{k_0}\supset\overline{\Omega}$.
\item The functions $\psi_{k,\n}$ are locally uniformly bounded in $\n$.
\end{enumerate}
\end{prop}

\begin{rem}
When working with a fixed $k_0$, we replace 
$\Omega'$ with $\Omega'\cap\Omega'_{k_0}$ to simplify notation.
\end{rem}

\begin{rem}
In the above construction, we only claim to obtain an orthonormal set
$(\psi_{k,\n})_k\subset A^2(\Omega, \e^{-2{\mathbf{Q}_\n}\cdot\Biglog(z)})$,
as opposed to a full orthonormal basis. 
Indeed, the space $A^2(\Omega, \e^{-2{\mathbf{Q}_\n}\cdot\Biglog(z)})$
can be larger in general. However, this is all we need for our proof.
\end{rem}

We start a lemma which ensures that there exists an orthonormal
basis $(\phi_k)_{k \in \N}$ of $A^2(\Omega)$ 
consisting of certain rational functions.
This is essentially an $L^2$-version of the 
Mergelyan theorem in a smooth setting, 
of which related versions have appeared previously 
(see e.g. Biard, Fornæss and Wu 
\cite{Fornaess} and the references therein).

\begin{lem}
\label{lem:density}
Let $\Omega$ be an $l$-connected planar domain  
with analytic boundary and let $z_1,\dots,z_l$ be fixed
points in each hole of $\Omega$. Then rational functions with poles at
$z_1, \dots,z_l$ form a dense subspace of 
$A^2(\Omega,\e^{-2{\bf Q}\cdot \Biglog(z)})$.
\end{lem}

\begin{proof}[Proof of Lemma~\ref{lem:density}]
Suppose first that $\Omega$ is bounded. 
We denote by $ \mathcal T_1,\dots,\mathcal T_l$
the holes of $\Omega$, i.e.,
the connected components of $\mathbb C \setminus \Omega$. 
We use the notation $\Omega_j
= \mathbb C \setminus \mathcal T_j$
for $j \in \{1,\dots, l\}$
and define also $\Omega_{l+1} = 
\Omega \cup \mathcal T_1 \cup \dots \cup \mathcal T_l$.
The domains $\Omega_1,\ldots,\Omega_{l+1}$ 
are simply connected (on the sphere), 
they have non-empty intersection with $\Omega$, and each
share a single boundary component with $\Omega$ and, 
in fact, $\partial\Omega$ is the disjoint union
$\bigcup_{j=1}^{l+1}\partial\Omega_j$.

Denote by $f$ an arbitrary element of 
$A^2(\Omega,\e^{-2{\bf Q}\cdot \Biglog(z)})$. 
By Cauchy's integral formula,
$f$ can be decomposed as 
\[
f(z)=\sum_{j=1}^{l+1}\frac{1}{2\pi\imag}\int_{C_j}\frac{f(w)}{w-z}\diff w=:
\sum_{j=1}^{l+1} g_j(z).
\]
where for $1\le j\le l+1$, the $C_j$ are small inward
perturbations towards $\Omega$ of the connected
components of $\partial\Omega$, taken with the standard orientation. 
By deforming the contours $C_j$, the functions $g_j$ are seen to be holomorphic
on $\Omega_j$.

We will show that $g_j$ can be approximated by certain rational functions.
We treat the cases $1\le j\le l$ and $j=l+1$ separately.

\subsubsection*{Case 1.}
We fix an arbitrary index $j$ with $1\le j\le l$.
Since for all $k\ne j$, the function $g_k$ is locally bounded in a neighborhood of 
$\partial\Omega_j$, 
it follows that $g_j=f-\sum_{k\ne j}g_k$ lies 
in $A^2(\Omega_j,|z-z_j|^{-2 \alpha})$ where
$\alpha$ is some arbitrary fixed strictly positive number. 
Indeed, near $\partial\Omega_j$, we use the uniform boundedness of $g_k$
for $k\ne j$ along with the fact that the $f$ lies in the 
space $A^2(\Omega,\e^{-2{\bf Q}\cdot \Biglog(z)})$,
and near infinity we use the facts that $|g_j|$ 
decays like $1/|z|$ and that $\alpha$ is
strictly positive.

There exists a conformal mapping $\phi_j:\Omega_j\to \D$ 
which extends holomorphically past the boundary, 
where we choose the normalization $\phi_j(\infty)=0$.
Below we will prove that $g_j$ can be arbitrarily well approximated in 
$A^2(\Omega_j,|z-z_j|^{-2\alpha})$ 
by functions of the form $\phi_j'\phi_j^k$ for appropriate powers $k\in\Z$. 
But then by Runge's theorem, since $\phi_j$ is holomorphic in
a neighborhood of $\Omega_j$, the function $\phi_j$ can be approximated in the 
uniform norm on a neighborhood of $\overline\Omega_j$ by rational 
functions with poles at $z_j$.
Finally, since $\Omega$ is bounded, the same 
statement holds in the norm in the space
$A^2(\Omega,\e^{-2{\bf Q}\cdot \Biglog(z)})$. 

It remains to show that $\{\phi'\phi_j^k:k\ge k_0\}$ spans the space 
$A^2(\Omega_j,|z-z_j|^{-2\alpha})$. 
By a change of variables with the conformal map, 
this follows from the fact that polynomials are dense in $A^2(\D,\omega)$ 
for weights which are smooth and positive in a 
neighborhood of $\partial\D$, which 
is equivalent to the classical statement that
polynomials are dense in $A^2(\D)$ (see e.g. the book \cite{HKZ}
by Hedenmalm-Korenblum-Zhu).

\subsubsection*{Case 2.} This case is similar, 
but we approximate $g$ by polynomials.

\bigskip

It only remains to treat the case when $\Omega$ is unbounded.
However, essentially the same proof works in that setting. 
The only difference is that
there is no curve $C_{l+1}$ around the point at infinity, and hence no function 
$g_{l+1}$ to approximate.
\end{proof}

Below, we need means to control 
basis elements of $A^2(\Omega,\rho)$ at infinity, in the case
when $\Omega$ is unbounded and $\rho$
is a positive weight with $\rho(z)\asymp |z|^\alpha$ as $z\to\infty$,
for some $\alpha\in\R$.
We use the following well-known fact on removable singularities.

\bigskip

\noindent \emph{Every $g:\D \setminus\{0\}
\to \mathbb C$ holomorphic such that
$\int_{\mathbb D}|g(z)|^2 \mathrm d m(z)< \infty$
admits a holomorphic extension to $\mathbb D$.}

\bigskip

In particular, for any $\gamma \in \mathbb R$,
there exists $\varepsilon>0$ such that
if $\int_{\mathbb D}|z|^\gamma |g(z)|^2
\mathrm d m(z) < \infty$ then
$|z|^\gamma |g(z)|^2 \leq C/|z|^{2-\varepsilon}$
for some constant $C>0$ and near $z = 0$.
Indeed, we may write
$\gamma = 2k + \xi$ with $\xi \in (-2,0]$ and
$k \in \mathbb Z$ so that
$\int_{\mathbb D} |z^k g(z)|^2
\leq 
\int_{\mathbb D} |z|^\gamma |g(z)|^2 < \infty$.
The above claim on removability of singularities 
implies that
$|z^k g(z)|^2$ is bounded around zero
which tells us that
\[|z|^\gamma |g(z)|^2 =
|z|^\xi |z^k g(z)|^2
\leq 
C |z|^\xi
=\frac{C}{|z|^{2-\varepsilon}}
\]
for some constants $C,\varepsilon >0$ and near $z = 0$.

From the previous argument, 
for a positive weight $\rho$ 
satisfying $\rho(z)\asymp |z|^\alpha$ as $z\to\infty$,
there exists $\varepsilon=\epsilon(\alpha) >0$ such that
if $f\in A^2(\Omega,\rho)$  then
\begin{equation}
\label{eq:holom-infinity}
|f(z)|^2\rho(z)\lesssim |z|^{-2-\varepsilon}.
\end{equation}
Indeed, for $g(z) =f(1/z)$ we would have
$\int |g(z)|^2 |z|^{-\alpha-4} \mathrm d m(z)<\infty$
which would imply that
$|g(z)|^2 |z|^{-\alpha-4} \leq C/|z|^{2+\varepsilon}$.
By going back to $f$,
this is precisely \eqref{eq:holom-infinity}.

\begin{proof}[Proof of Proposition~\ref{prop:good-basis}]
Let ${\bf Q}_\n$ be the sequence of elements of 
$(\R /\Z)^l$ which converge towards
${\bf Q} \in [0,1)^l$. 
We first notice that for $\n$ large enough
\[
A^2(\Omega,\e^{-2{\bf Q}\cdot \Biglog(z)})
\subset A^2(\Omega,\e^{-2{\bf Q}_\n\cdot \Biglog(z)}).
\]
Indeed, if $\Omega$ is bounded, this is simply the fact that $\Biglog(z)$
is bounded on $\Omega$. If 
instead $\Omega$ contains a neighborhood of infinity,
\eqref{eq:holom-infinity}
implies that for any $f\in A^2(\Omega,\e^{-2{\bf Q}\cdot \Biglog(z)})$,
\[
|f(z)|^2\e^{-2{\bf Q}
\cdot \Biglog(z)}\le C|z|^{-2-\varepsilon},
\]
for a fixed $\epsilon$,
so that since $\|{\bf Q} - {\bf Q}_\n\|_1 < \varepsilon$ for $\n$ large enough,
$f$ belongs to the space $A^2(\Omega,\e^{-2{\bf Q}_\n\cdot \Biglog(z)})$ as well.

Denote by $(\psi_k)_{k \in \N}$ an orthonormal basis of 
$A^2(\Omega,\e^{-2{\bf Q} \cdot \Biglog z})$
consisting of rational functions with poles at the $z_j$, 
the existence of which is
guaranteed by Lemma~\ref{lem:density}.
We define $(\psi_{k,\n})_{k \in \N}$ as the orthonormal set
obtained from the Gram-Schmidt process applied to  $(\psi_k)_{k \in \N}$
in the topology of $L^2(\Omega,\e^{-2{\bf Q}_\n \cdot \Biglog (z)})$.
All these functions are defined on $\C \setminus \{z_1, \dots, z_l\}$.

It only remains to establish the convergence
\[
\psi_{k,\n}\to \psi_k
\]
as $\n\to \infty$ and the locally uniform boundedness in $\n$ of $\psi_{k,\n}$. 
But this would be immediate if we can show that the Gram
matrix $G_\n=(\langle \psi_k,\psi_m
\rangle_{L^2(\Omega,\omega_{{\bf Q}_\n})})_{1\le j,k\le k_0}$
satisfies $G_\n=I+\ordo(1)$ as $\n\to\infty$, using that
the functions $\psi_k$ are locally uniformly bounded for the last point. 
But the statement about the Gram matrices, in turn, follows from the bound 
\eqref{eq:holom-infinity}, the convergence ${\bf Q}_\n\to{\bf Q}$ and the 
dominated convergence theorem. 
The proof is complete.
\end{proof}

\subsection{Quasipolynomials and \texorpdfstring{$\bar\partial$}{asd}-surgery}
Let $\psi_{k,\n}$
be the sequence
from Proposition~\ref{prop:good-basis}
for ${\bf Q}_\n = \class{\kappa_\n {\bf q}}$. 
We use it to construct an orthonormal basis 
$\big(P_{k,\n}\big)_{0\le k \leq \n-1}$
of the polynomial Bergman space 
$A^2_\n(\Omega,\e^{-2\kappa_\n U^\mu})$ which satisfy, for any fixed $k$,
\[
\forall z \in \Omega, \quad \lim_{\n\to\infty}|P_{k,\n}(z)|
\e^{-\kappa_\n U^{\mu}(z)} =
|\psi_k(z)| \e^{-{\bf Q} \cdot \Biglog (z)}.
\]
Denote by $\chi$ a smooth cut-off function which vanishes
in a neighborhood of $\C\setminus\overline{\Omega'}$, and
is identically one in a neighborhood of $\overline\Omega$.
See Figure~\ref{fig:cut-off}.

\begin{defn}
[Approximately orthogonal holomorphic functions]
For any $\n,k \in \N$ we define
the function
$F_{k,\n}:\C \to \C$ by
\begin{equation}\label{approx}
F_{k,\n}(z) = \chi(z) \psi_{k,\n}(z)
\e^{\kappa_\n \mathfrak h(z)} \prod_{i=1}^l
(z-z_i)^{\kappa_\n q_i-\class{\kappa_\n q_i}}.
\end{equation}
\end{defn}

Notice that
$
|F_{k,\n}(z)|
=
|\psi_{k,\n}(z)|\, \e^{\kappa_\n U^\mu(z)}
\e^{-\mathbf Q \cdot \Biglog(z)}
$ for $z \in \Omega$.
The following lemma shows that the $F_{k,\n}$ are
approximately orthogonal in the space $L^2(\C,\e^{-2\kappa_\n U^{\mu}})$.
In addition, they are close to polynomials in $\C_{\n-1}[z]$ 
in the suitable $L^2$-sense
(see Proposition~\ref{prop:hormander} below).
In the case when $\Omega$ is unbounded, they 
behave like elements of $\C_{\n-1}[z]$ at infinity, which is certainly
necessary for accurate polynomial approximation to be possible.

\begin{lem}
For any $\n,k \in \N$ , the function $F_{k,\n}$
belongs to $L^2(\C,\e^{-2\kappa_\n U^{\mu}})$. Moreover,
for any fixed $k,m \in \mathbb N$, it holds that
\[
\langle F_{k,\n},F_{m,\n}\rangle_\n
= \int_{\C} F_{k,\n}(z) \bar F_{m,\n}(z)
\e^{-2\kappa_\n U^{\mu}(z)}
\diffA(z)
\xrightarrow[\n \to \infty]{} \delta_{k,m}.
\]
\end{lem}

\begin{figure}[t!]
\centering
\includegraphics[width=.64\linewidth]{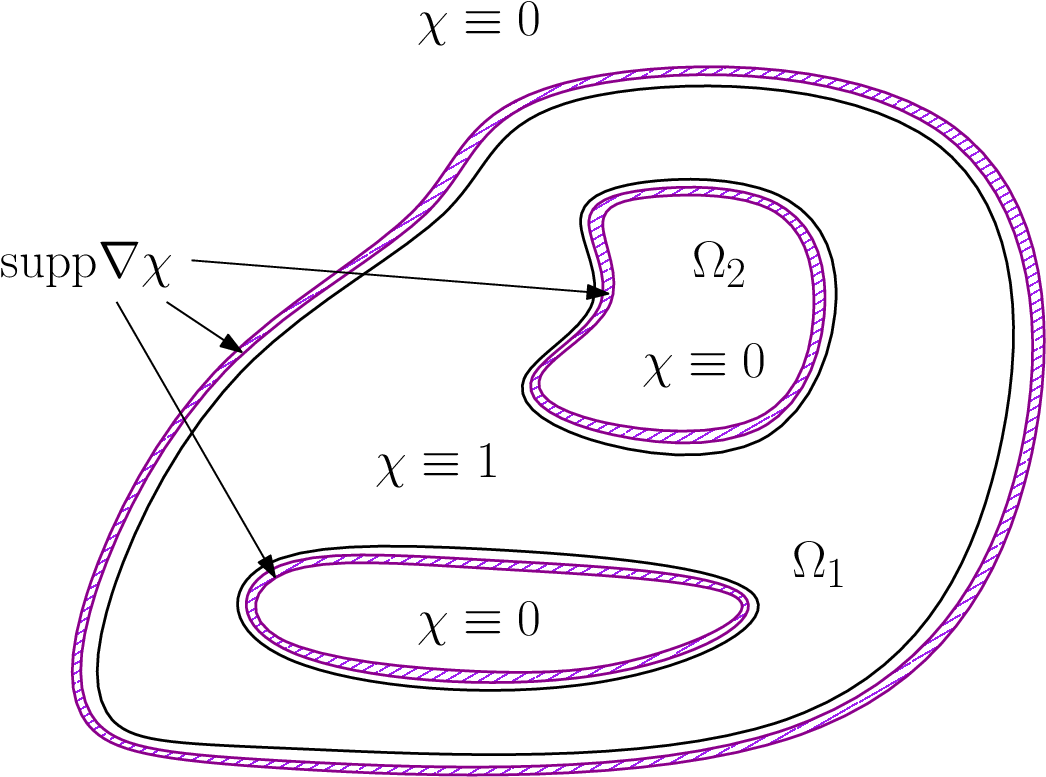}
\caption{The cut-off function $\chi$ for $\Omega=\Omega_1$ 
used to define the function $F_{k,\n}$.}
\label{fig:cut-off}
\end{figure}

\begin{proof}
By the definition of $\psi_{k,\n}$ we find that
\[
\int_\Omega F_{k,\n}(z) \bar{F}_{m,\n}(z)
\e^{-2\kappa_\n U^{\mu}(z)}
\diff \m_\C(z)
= \delta_{k, m}.
\]
On the set $\Omega' \setminus \overline{\Omega}$,
Proposition~\ref{prop:extension1} implies that
$U^{\mu}-\mathcal{U}^\mu>0$. In view of Proposition~\ref{prop:good-basis} part 
(iii), the functions $\psi_{k,\n}$ are locally uniformly 
bounded in $\n$, so the same holds
for the functions $|F_{k,\n}|^2\e^{-2\kappa_\n U^\mu}$. We may hence
apply the dominated convergence theorem to
obtain
\begin{multline}
\int_{\Omega^c \setminus \Omega} |F_{k,\n}(z)|^2
 \e^{-2\kappa_\n U^{\mu}(z)}
\diff \m_\C(z)
\\
= \int_{\Omega^c \setminus \Omega} \chi(z)|\psi_{k,\n}(z)|^2
\e^{-2\kappa_\n (U^{\mu}(z)- \calU^{\mu}(z))}
\e^{-2 \mathbf Q \cdot \Biglog(z)}\to 0
\end{multline}
as $\n$ tends to infinity. Hence, we can conclude that
\[
\langle F_{k,\n}, F_{m,\n}
\rangle_\n
\xrightarrow[\n \to \infty]{} \delta_{k, m}.
\]
This completes the proof.
\end{proof}

Turning to the $\bar\partial$-surgery,
we have the following.

\begin{prop}[$\bar\partial$-surgery and H\"ormander estimates]
\label{prop:hormander}
For any fixed $k \in \N$, there exists a sequence of smooth functions
$(v_{k,\n})_{\n\in \N}$ defined on $\C$ with
$\|v_{k,\n}\|_\n^2 
\to 0$
as $\n\to\infty$, such that
$Q_{k,\n} :=F_{k,\n}-v_{k,\n}$
is a polynomial of degree at most $\n-1$ and
we have the locally uniform convergence
\[
Q_{k,\n}(z)\e^{-\kappa_\n \mathfrak h(z)}
\prod_{i=1}^l
(z-z_i)^{-(\kappa_\n q_i-\class{\kappa_\n q_i})} 
\xrightarrow[\n \to \infty]{} \psi_k(z)
\]
on $\Omega$.
\end{prop}

This claim is essential for our proof. It follows from
H{\"o}rmander's classical $L^2$-estimate for the 
$\bar\partial$-operator, with some 
minor adjustments which we proceed to outline. The first adjustment 
(this version is taken from \cite{AHM}) gives 
control of the growth of the solution.
Recall that a $C^{1,1}$ function
is a differentiable function whose
derivative is locally Lipschitz.

\begin{lem}[H\"ormander estimate]
\label{lem:horm}
Let $\phi$ be a $C^{1,1}$ subharmonic function on
$\C$ such that $\phi(z)\geq (1+c)\log|z|^2$
for some constant $c >0$
and for $|z|$ large enough.
Let $\mathcal{K} \subset \mathbb C$ be
a compact set
where $\phi$ is strictly subharmonic.
For any $f \in L^\infty(\mathbb C)$ supported
on $\mathcal{K}$, there exists a solution 
$u: \C \to \C$ to the equation $\bar\partial u=f$
that satisfies
\begin{equation}
\int_{\C} |u|^2\e^{-\phi}\diff\m \le
2\int_\mathcal{K}|f|^2\frac{\e^{-\phi}}{\Delta\phi}\diff\m.
\end{equation}
\end{lem}

\begin{proof}
This follows directly from H{\"o}rmander's 
classical result for a $C^{1,1}$ function $\phi$
that satisfies $\Delta\phi>0$ on $\C$ (which gives 
the the conclusion of
Lemma~\ref{lem:horm} without the factor
2 on the right-hand side). 
Indeed, the function $\phi: \mathbb C \to \mathbb R$
defined by
\begin{equation}
\phi_\epsilon(z)=(1-\epsilon)\phi(z)+\epsilon\log(1+|z|^2).
\end{equation}
 meets
the conditions of H{\"o}rmander's theorem.
Then, for any $\epsilon>0$ there exists a 
solution $u=u_\epsilon$ to the equation
$\bar\partial u=f$, and when $\epsilon$ is sufficiently
small, we obtain
\[
\int |u|^2\e^{-\phi_\epsilon}\diffA\le \int |f|^2\frac{\e^{-\phi}}{\Delta\phi}
\frac{\Delta \phi}{\Delta \phi_\epsilon}
\e^{\epsilon(\phi-\log(1+|\cdot|^2))}\diffA
\le\frac32\int_{\mathcal{K}} |f|^2\frac{\e^{-\phi}}{\Delta \phi}\diffA.
\]
To complete the proof, it is enough to show that
for small $\epsilon >0$, we have
\[
\int |u|^2\e^{-\phi}\diffA\le \frac43\int |u|^2\e^{-\phi_\epsilon}\diffA.
\]
For this, we can use that,
for $|z|$ large enough,
\begin{align*}
\phi(z)&=\phi_\epsilon(z) +\epsilon(\phi(z)-\log(1+|z|^2))
														\\
&\ge \phi_\epsilon(z)+
\epsilon(\phi(z)-(1+\tfrac{c}{2}) \log|z|^2)			
														\\	
&\ge \phi_\epsilon(z) + \epsilon\tfrac{c}{2}\log|z|^2.
\end{align*}
We may thus split this integral according to 
$\C=\D(0,R)\cup(\C\setminus\D(0,R))$,
and use $\phi_\epsilon= \phi+\ordo(1)$ on the bounded set and
$\phi\ge\phi_\epsilon+\Ordo(1)$ on the unbounded set.
This completes the proof.
\end{proof}

Since $U^\mu$ is not necessarily $C^{1,1}$, 
we cannot apply Lemma~\ref{lem:horm} as stated with the
potential $\phi =2\kappa_\n U^{\mu}$ directly, so we introduce a new
potential $V$ which satisfies the required regularity assumptions,
whose existence and properties are given in the following lemma.
Before proceeding, the reader may want to recall 
Propositions~\ref{prop:extension1} and \ref{prop:extension2}.

\begin{figure}[t!]
\centering
\includegraphics[width=.65\linewidth]{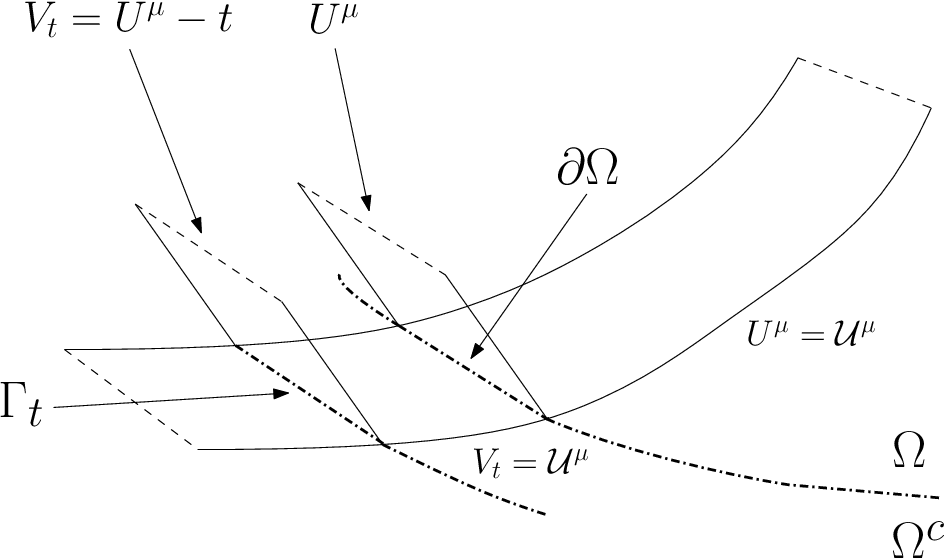}
\caption{Illustration of the functions 
$V_t$, $U^\mu$ and $\mathcal{U}^\mu$ from the proof of Lemma~\ref{lem:changeofV}
near $\partial\Omega$ and $\Gamma_t$.}
\label{fig:Vt-construct}
\end{figure}

\begin{lem} \label{lem:changeofV}
There exists a $C^{1,1}$-smooth subharmonic function 
$V:\mathbb C \to \mathbb R$ and a domain 
$\Omega''$ with $\overline{\Omega}\subset\Omega''\subset\Omega'$
such that
\begin{enumerate}[label=(\roman*)]
\item the inequality $V(z) \leq U^{\mu}(z)$ holds for $z\in\C$,
\item for $z \in \Omega$ we have equality $V(z) = U^{\mu}(z)$,
\item for $z \in \Omega''$ it holds that $\mathcal{U}^\mu(z)\le V(z)$, and
\item there exists $\epsilon_1,\epsilon_2>0$ such that 
$\Delta V(z)\ge \delta>0$ whenever
$\epsilon_1\le \mathrm{d}_{\C}(z,\Omega)\le \epsilon_2$.
\end{enumerate}
\end{lem}

The idea of the proof is to construct 
\say{inverse balayage potentials} $V_t$ obtained
by moving the mass $\mu\vert_{\partial\Omega}$ 
to curves $\Gamma_t$ inside $\overline{\Omega}^c$, and then 
mollify $\Delta V_t$ to obtain an averaged inverse balayage potential $V$.
It may be instructive to consider first the simple 
case when $\mu=\mu_1$ is normalized arc-length 
measure on the unit circle, in which case $U^\mu$ equals
$\log|z|$ outside $\D$ and vanishes on $\Omega=\D^c$.
By Newton's theorem, the logarithmic potential 
of any rotationally invariant probability measure
supported in $\D$ has the same logarithmic potential as $\mu$ outside $\D$. 
If we denote by $\mu_r$ the normalized arc-length measure on $\D(0,r)$,
we see that $U^{\mu_r}=U^\mu$ outside $\D$ for $r<1$ and $U^{\mu_r}$ 
supplies an inward harmonic continuation of 
$U^{\mu}\vert_{\D^c}$. As a consequence, 
a function $V$ with the desired properties
can be obtained by a weighted average 
$V(z)=\int \varphi(t)U^{\mu_t}(z)\diff t$,
where $\varphi$ is a sufficiently regular 
non-negative function compactly supported in $(0,1]$
with $\int\varphi(t)\diff t=1$.

\begin{proof}
We will make use of the harmonic extension $\calU$ of 
$U^\mu\vert_\Omega$ past the 
boundary $\partial\Omega$ to some larger domain, denoted $\Omega'$. 
By performing a Taylor expansion at $\partial\Omega$, 
we found in Proposition~\ref{prop:extension2} 
that $(U^\mu-\calU)(z_0-t{\rm n})= c_1(z_0)t + \Ordo(t^2)$ 
or $(U^\mu-\calU)(z_0-t{\rm n})= c_2 t^2 + \Ordo(t^3)$ depending on whether 
$\mu$ charges the relevant boundary component.
Here $c_i(z)>0$ are smooth strictly positive functions on the components of 
$\partial\Omega$, 
$t>0$, and ${\rm n}$ denotes the unit normal pointing into $\Omega$
at $z_0\in\partial\Omega$. 

It suffices to prove the lemma when the density $\densityBulk$ 
of $\mu$ with respect to area measure vanishes identically.
Indeed, if $\mu_2$ is
the part of the measure $\mu$ which is supported on $\partial\Omega$,
we may apply the lemma to $U^{\mu_2}$ to produce a 
function $V_2$. Then since $U^{\mu_1}$
is $C^{1,1}$-smooth, the function $V=U^{\mu_1}+V_2$
satisfies the requirements of the lemma.
We moreover assume that $\Omega$ is unbounded -- the 
bounded case also follows in entirely the same fashion.
It is similarly clear that the claim is local near each boundary component, 
so we may without loss of generality assume that 
$\Omega$ is simply connected on the sphere.

Recall that $D$ is the support of the measure $\mu$.
By taking $t_0>0$ and $d_0$ small enough, 
the level curves 
\[
\Gamma_{t}=\{z\in D: \mathrm{d}(z,\partial\Omega)\le d_0, 
U^\mu-\calU^\mu=t\}
\] 
remain smooth and simple for $0\le t\le 2t_0$, and deform smoothly with $t$. 
This follows e.g.\ from a direct construction of conformal maps $\varphi_t$
of $\D_\e$ onto the domains enclosed by the level curves 
which extend conformally onto a larger exterior 
disk \cite[\S~6.2]{HW}, but it can also be seen by more elementary means using
Morse Theory\footnote{Specifically the regular interval lemma, see, e.g., 
Hirsch's book \cite[Theorem~2.2]{Hirsch}.}.
We take $\Omega''$ to be the unbounded
component of $\C\setminus\Gamma_{2t_0}$. 
The set $\Omega''\setminus\Omega$ is a thin annular
domain in $D$ next to $\partial\Omega$, 
which is foliated by the level curves $\Gamma_t$
for $0\le t\le 2t_0$.
Then, for each $t$ with $0<t<t_0$, we define the function $V_t$ 
by declaring that it equals $U^\mu-t$ on the region $D_t$ 
enclosed by $\Gamma_t$ and $\calU^\mu$ outside $D_t$, 
see Figure~\ref{fig:Vt-construct}. Hence, we have
$V_t\le U^\mu$ on $\C$, $V_t=U^\mu$ on $\Omega$, and on any 
level curve $\Gamma_{t'}$ with $t<t'\le 2t_0$ we have
\[
V_t-\calU^\mu=U^\mu-t-\calU^\mu=t'-t>0.
\]

We now construct the function $V$ by integration
\begin{equation}
\label{eq:VasIntegralPhi}
V(z)=\frac{1}{t_0}\int_{0}^{t_0}V_t(z)\diff t.
\end{equation}
This construction mollifies the measures $\Delta V_t$, so that
the desired comparisons between $U^\mu$, $\mathcal{U}^\mu$
and $V$ remain valid. To see that $V\in C^{1,1}$ 
and to calculate its Laplacian, note
that for any smooth compactly supported
function $f$ we have, 
by Fubini's theorem and uniform boundedness of $(V_t)_t$, that
\begin{align}
\label{eq:dist-calcLapV1}
\int \Delta f(z)V(z)\diffA(z)&=
\frac{1}{t_0}\int_{0}^{t^0}\int \Delta f(z)V_t(z)\diffA(z)\diff t
\\&=\frac{1}{t_0}\int_{0}^{t^0}\int f(z)\Delta V_t(z)\diffA(z)\diff t
\end{align}
But, in the sense of distributions, 
$\Delta V_t$ equals the Neumann jump of $V_t$ across $\partial\Omega$:
\begin{equation}
\label{eq:Neumann-jump}
\mathcal{N}_{\Gamma_t}[V_t]\diffs_{\Gamma_t}:=
-\Big(\partial_{{\rm n}_{D_t}}+\partial_{{\rm n}_{D_t^c}}\Big)V_t=
\partial_{{\rm n}_{D_t}}(\mathcal{U}^\mu-U^\mu),
\end{equation}
(see, e.g., \cite[Eq. (2.5)]{AHM3}), 
where the last equality holds
since $\mathcal{U}^\mu$ is smooth across $\Gamma_t$. Here, ${\rm n}_{D_t}$ 
and ${\rm n}_{D_t^c}$ denote the outward unit normals
to $D_t$ and $D_t^c$ at $\Gamma_t$, respectively.
As a consequence of \eqref{eq:Neumann-jump}, we find that
\begin{align}
\label{eq:dist-calcLapV}
\frac{1}{t_0}\int_0^{t_0}\int \Delta f(z) V_t(z)\diffA\diff t&
=\frac{1}{t_0}\int_0^{t_0}\int_{\Gamma_t} 
f(z)\mathcal{N}_{\Gamma_t}[V_t](z)\diffs(z)\diff t
\\&=\frac{1}{t_0}\int_0^{t_0}\int_{\Gamma_t}f(z)
\partial_{{\rm n}_{D_t}}(\mathcal{U}^\mu-U^\mu)(z)\diffs(z)\diff t
\\&=\frac{1}{t_0}\int_{D\setminus D_{t_0}}f(z)
|\nabla(\mathcal{U}^\mu-U^\mu)(z)|^2\diffA(z)
\end{align}
where the last step follows from the co-area formula, which asserts that 
$\int_{D\setminus D_{t_0}} F\diffA=\int_0^{t_0}\int_{\Gamma_t} 
F/|\nabla(\mathcal{U}^\mu-U^\mu)|\diffs\diff t$, and the fact that
$\nabla (\mathcal{U}-U^\mu)$ is orthogonal to the level 
sets $\Gamma_t$ of $U^\mu-\mathcal{U}$. 
Combining \eqref{eq:dist-calcLapV} with \eqref{eq:dist-calcLapV1}, we see that 
\[
\Delta V=t_0^{-1}|\nabla (\mathcal{U}^\mu-U^\mu)|^2 1_{D\setminus D_{t_0}}
\]
in the sense of distributions and since 
$|\nabla (\mathcal{U}^\mu -U^\mu)|>0$ on $\partial\Omega$, 
and hence we find that $\Delta V$ meets the desired positivity
condition $(iv)$ for $t_0$ sufficiently small. Moreover, it is clear that 
$V$ is $C^{1,1}$-smooth. Indeed, a computation
similar to \eqref{eq:dist-calcLapV} reveals that the gradient 
$\nabla V$ is Lipschitz continuous and, in fact, piecewise analytic.
Since $V$ also satisfies $(i)$-$(iii)$ by construction, the proof is complete.
\end{proof}

By possibly shrinking $\Omega'$, we may 
assume without loss of generality that $\Omega'=\Omega''$.

We are now ready for the proof of our proposition on $\bar\partial$-surgery.
\begin{proof}[Proof of Proposition~\ref{prop:hormander}]
By Lemma~\ref{lem:changeofV}, there
exists a $C^{1,1}$ function $V:\mathbb C \to \mathbb R$
and a constant $\delta>0$ such that
\[
V \geq
\calU^{\mu}  \text{ and }
\Delta V \geq \delta
\text{ where } \nabla \chi \neq 0,\]
and such that
$U^\mu \geq V$ everywhere and
$V\vert_\Omega = U^{\mu}\vert_\Omega$.
By Lemma~\ref{lem:horm} applied with $\phi = 2\kappa_\n V$ and
$f=\bar\partial F_{k,\n}$, that is
\[
f = (\bar{\partial} \chi ) \psi_{k,\n}
\e^{\kappa_\n \mathfrak h} 
\prod_{i=1}^{l}(z-z_i)^{\kappa_\n q_i-\class{\kappa_\n q_i}},
\]
there exists a function $v_{k,\n}$ such that
$\bar \partial  v_{k,\n}=\bar \partial F_{k,\n}$ and
\begin{align}
\label{eq:horm-pf-eq1}
\int_\C \big|v_{k,\n}\big|^2 \e^{-2\kappa_\n V}
\diff \m 
& \le \int_{\supp \nabla \chi}
\frac{\big|\bar \partial F_{k,\n}\big|^2}{\kappa_\n \Delta V}
\e^{-2\kappa_\n V}\diff\m \\
&\le
\int_{\supp \nabla \chi}
\frac{\big|\bar \partial \chi \big|^2}{\kappa_\n \delta}
\big|\psi_{k,\n} \big|^2
\e^{-2\kappa_\n \left(V- \mathcal U^\mu \right)}
\e^{2\mathbf q \cdot \Biglog}
\diff\m.
\end{align}

The important thing to notice is that on the support of $\bar\partial \chi$, 
(which is essentially the gradient $\nabla\chi$)
the weight $\e^{-2\kappa_\n(V-\calU)}$ is bounded, and $\psi_{k,\n}$
is uniformly bounded in $\n$.
Hence
the right-hand side of \eqref{eq:horm-pf-eq1} tends to zero as $\n$ goes to
infinity, and since $V\le U^\mu$ on $\C$ we have
\[
\| v_{k,\n} \|^2_\n=
\int_\C \big|v_{k,\n}\big|^2 \e^{-2\kappa_\n U^\mu} \diff\m
\leq
\int_\C \big|v_{k,\n}\big|^2 \e^{-2\kappa_\n V} \diff\m
\xrightarrow[\n \to \infty]{} 0.
\]
In particular, this shows that if we put 
$Q_{k,\n} := F_{k,\n}-v_{k,\n}$, 
then $\lVert Q_{k,\n}\rVert_{N}<+\infty$.
The bound \eqref{eq:holom-infinity} implies that there
exists a constant $C>0$ such that
\[
\big|Q_{k,\n}(z)\big| \leq
\frac{C |z|^{\kappa_\n}}{|z|^{1+\epsilon/2}}
\leq \frac{C |z|^{\n+1}}{|z|^{1+\epsilon/2}}\leq C |z|^{\n-\epsilon/2}
\]
for $|z|$ large enough.
Since $\bar \partial Q_{k,\n} 
= \bar \partial F_{k,\n} -
\bar \partial v_{k,\n}=0$, we get that
$Q_{k,\n}$ is entire and,
since $|Q_{k,\n}|$
grows at most like
$|z|^{\n-\epsilon/2}$ (in fact like $|z|^{\n-1}$),
Liouville's theorem implies that
$Q_{k,\n}$ is a polynomial of degree at most $\n-1$.

The holomorphicity of $v_{k,\n}|_\Omega$
implies that
$\big|v_{k,\n} \big|^2
\e^{-2\kappa_\n U^\mu}$ is subharmonic 
on $\Omega$. Then, for any fixed $z \in \Omega$ and any $\varepsilon>0$
such that $\bar D(z,\varepsilon) \subset \Omega$,
we may write
\begin{align*}
\big|v_{k,\n}(z) \big|^2
\e^{-2\kappa_\n U^\mu(z)}
& \le 
\frac{1}{\pi \varepsilon^2} \int_{D(z,\varepsilon)}
\big|v_{k,\n}(z) \big|^2
\e^{-2\kappa_\n U^\mu(z)}
\diff \m_{\mathbb C}(z)
\\ & \le \frac{1}{\pi \varepsilon^2} \|v_{k,\n}\|_\n^2 
\xrightarrow[\n \to \infty]{} 0.
\end{align*}
This implies that 
\[
\Big| Q_{k,\n}(z)\e^{-\kappa_\n \mathfrak h(z)}
\prod_{i=1}^l
(z-z_i)^{-(\kappa_\n q_i-\class{\kappa_\n q_i})} 
- \psi_k(z) \Big|^2
=
\big|v_{k,\n}(z) \big|^2
\e^{-2\kappa_\n U^\mu(z)}
\xrightarrow[\n \to \infty]{} 0
\]
locally uniformly, which is what we wanted.
\end{proof}

Recall that, by \eqref{eq:holom-infinity} and Liouville's theorem,
 $A^2(\C,\e^{-2\kappa_\n U^\mu})=\C_{\n-1}[z]$ as sets.

\begin{lem}[The adapted orthonormal basis]\label{lem:goodpolybasis}
For any fixed $k_0 \in \N$, 
there exists an orthonormal family  $(P_{k,\n})_{k \leq k_0}$
in $A^2(\C,\e^{-2\kappa_\n U^\mu})$, such that
for any $k \leq k_0$, we have
\[
P_{k,\n}(z)\e^{-\kappa_\n \mathfrak h(z)}
\prod_{i=1}^l
(z-z_i)^{-(\kappa_\n q_i-\class{\kappa_\n q_i})} 
\xrightarrow[\n \to \infty]{} \psi_k(z).\]
for every $z \in \Omega$. In particular,
\[
\big|P_{k,\n} (z) \big|^2
\e^{-2\kappa_\n U^{\mu}(z)}
\xrightarrow[\n \to \infty]{}
\big|\psi_k(z) \big|^2
\e^{-2\mathbf Q \cdot \Biglog(z)}.
\]
\end{lem}

\begin{proof}
Recall the holomorphic functions $F_{k,\n}$ and the 
$\bar\partial$-corrections $v_{k,\n}$ from Proposition~\ref{prop:hormander}.
Since
\[\langle  F_{k,\n}, F_{m,\n} \rangle_\n
\xrightarrow[\n \to \infty]{} \delta_{k, m}\quad \text{ and } \quad
\|v_{k,\n}\|_\n \xrightarrow[n\to \infty]{} 0,\]
we have
\begin{equation}
\label{eq:AlmostOrthogonal}
\langle Q_{k,\n},
Q_{m,\n}\rangle_\n
\xrightarrow[\n \to \infty]{}
\delta_{k, m},
\end{equation}
where $Q_{k,\n}$ are the polynomials 
$Q_{k,\n}=F_{k,\n}-v_{k,\n}$
from Proposition~\ref{prop:hormander}.
In other words, for any fixed positive integer $k_0$,
the Gram matrix 
$(\langle Q_{k,\n},
Q_{m,\n}\rangle_\n)_{1 \leq k,m \leq k_0}$ of the family 
$(Q_{k,\n})_{k \leq k_0}$
converges towards the identity matrix, hence for $\n$ large enough, it is
invertible and the family 
$(Q_{k,\n})_{1\leq k\leq k_0}$ is linearly independent.
Moreover, if $(P_{k,\n})_{k \leq k_0}$ is the 
Gram-Schmidt orthonormalization of
$(Q_{k,\n})_{k \leq k_0}$, the fact
that the Gram matrix 
of $(Q_{k,\n})_{k \leq k_0}$ converges to the identity
let us conclude the proof.
This can be seen, for instance, by induction or
by using an explicit formula
such as the following.
Fix $k \le k_0$ and let $D_{k}^{(\n)}$ 
be the the determinant of the Gram matrix of 
$(Q_{j,\n})_{j \leq k}$ and let $D_{k-1}^{(\n)}$ be 
the determinant of the Gram matrix of 
$(Q_{j,\n})_{j \leq k-1}$. Then, 
\[P_{k,\n}(z) =
\frac{1}{\sqrt{D_{k}^{(\n)} 
D_{k-1}^{(\n)}}}
\left| \hspace{-2mm}
\begin{array}{cccc}
\langle Q_{1,\n},Q_{1,\n} \rangle_\n &
\langle Q_{2,\n},Q_{1,\n} \rangle_\n &
\dots & \langle Q_{k,\n},Q_{1,\n} \rangle_\n
\\
\langle Q_{1,\n},Q_{2,\n} \rangle_\n &
\langle Q_{2,\n},Q_{2,\n} \rangle_\n &
\dots & \langle Q_{k,\n},Q_{2,\n} \rangle_\n				
\\
\vdots & \vdots & \ddots & \vdots
\\
\langle Q_{1,\n},Q_{k-1,\n} \rangle_\n &
\langle Q_{2,\n},Q_{k-1,\n} \rangle_\n &
\dots & \langle Q_{k,\n},Q_{k-1,\n} \rangle_\n			
\\
Q_{1,\n}(z) & Q_{2,\n}(z) & \dots & Q_{k,\n}(z)
\end{array}\hspace{-2mm}
\right|.
\]
Using the linearity of the determinant
in the last row and the convergence
of the inner products, we see that
$P_{k,\n}=Q_{k,\n}(1+\ordo(1))$.
Proposition~\ref{prop:hormander} then allows us to conclude
the proof by multiplying with 
$\e^{-\kappa_\n \mathfrak h(z)}
\prod_{i=1}^l
(z-z_i)^{-(\kappa_\n q_i-\class{\kappa_\n q_i})} $.
\end{proof}
\section{Proofs of the main results}
\subsection{Proof of Theorem~\ref{thm:NonSimplyConnected}}
We finish the proof of the theorem by completing the family
$(P_{k,\n})_{k \le k_0}$ to an orthonormal basis
$(P_{k,\n})_{k \le \n-1}$ of $A^2(\C,\e^{-2\kappa_\n U^\mu})$.
We obtain
\[
\sum_{k=0}^{k_0}
\big|P_{k,\n}(z)\big|^2
\e^{-2\kappa_\n U^{\mu}(z)}
\leq  \sum_{k=0}^{\n-1}
\big|P_{k,\n} (z) \big|^2
\e^{-2\kappa_\n U^{\mu}(z)}
=\kerK_\n(z,z)
\]
and, since
\[\sum_{k=1}^{k_0}
\big|\psi_k(z) \big|^2
\e^{-2\mathbf Q \cdot \Biglog(z)}
=
\lim_{\n \to \infty}
\sum_{k=1}^{k_0}
\big|P_{k,\n} (z) \big|^2
\e^{-2\kappa_\n U^{\mu}(z)},
\]
we obtain for any $z \in \Omega$
\[\sum_{k=1}^{k_0}
\left|\psi_k(z) \right|^2
\e^{-2\mathbf Q \cdot \Biglog(z)}
\leq \liminf_{\n \to\infty} \kerK_\n(z,z).
\]
By taking $k_0\to \infty$, we find that
\[
\BergK_{\Omega, \bf Q}(z,z) = \sum_{k=1}^\infty
\left|\psi_k(z) \right|^2
\e^{-2\mathbf Q \cdot \Biglog(z)}
\leq \liminf_{\n \to\infty} \kerK_\n(z,z) 
\]
and in view of the convergence $\class{\kappa_{\n}{\bf q}}\to{\bf Q}$,
Lemma~\ref{lem:monotone} and Proposition~\ref{thm: bergman continuity}
together imply that $\limsup_{N\to\infty} \kerK_\n(z,z)
\le B_{\Omega, \bf Q}(z,z)$,
which completes the proof (recall that all limits are taken along 
the subsequence $\mathcal{L}$).

\subsection{Proof of Theorem~\ref{thm:main-simple}}
If $\Omega$ is a bounded simply connected domain, then the theorem
is an immediate application of Theorem~\ref{thm:NonSimplyConnected}.
If $\Omega$ is the unbounded component of
$\C \setminus \supp(\mu)$ and is simply connected in
$\widehat{\C}$ (for instance, the complement of a single Jordan domain),
then we apply Theorem~\ref{thm:NonSimplyConnected},
and we remark that the $\mu$-mass
of the hole is $q_1=1$, hence $(\n+1)q_1$ it is an integer and the
sequence $\class{\n+1}$ is identically equal to zero.

\subsection{Proof of Theorem~\ref{thm:indep}}

The proof of the independence between components is  general and elementary.
Recall from the proof of
Theorem~\ref{thm:NonSimplyConnected}
that,
in view of
\eqref{eq:conv-masses},
the kernels
associated with
$\Phi_{1,\n}$ and 
$\Phi_{2,\n}$ (called
here
$\kerK_{1,\n}$ and 
$\kerK_{2,\n}$)
converge uniformly on compact 
subsets of $\Omega_1 \times \Omega_1$ and
$\Omega_2 \times \Omega_2$ towards two limiting Bergman kernels, 
which we denote by $\BergK_1$ and $\BergK_2$, respectively.
If we show that the kernel of the determinantal point process
\[
\Phi_{\n,\Omega_1 \cup \Omega_2} =  
\Theta_\n \cap \big(\Omega_1 \cup \Omega_2\big)
\]
converges uniformly on compact subsets of 
$(\Omega_1 \cup \Omega_2) \times (\Omega_1 \cup \Omega_2)$
towards the kernel
\[
\BergK (z,w) \overset{\rm def}{=} \begin{cases}
\BergK_1(z,w) & \text{ if  } (z,w) \in \Omega_1\times \Omega_1, \\
\BergK_2(z,w) & \text{ if } (z,w) \in \Omega_2\times \Omega_2, \\
0  &\text{ if } (z,w) \in (\Omega_1 \times \Omega_2) \cup (\Omega_2 
\times \Omega_1),
\end{cases}
\]
then the asymptotic independence of the point processes
follows (see Proposition~\ref{prop:UCVofKernels}) since
$\BergK $ corresponds to the kernel of the 
union of independent determinantal point
processes associated with $\BergK_1$ and $\BergK_2$
(see Proposition~\ref{prop:UnionOfDPP} of the Appendix). 

Here, as in Section~\ref{s:Bergman},
the kernel $\holkerK_\n$ associated with
$\Phi_{\n,\Omega_1 \cup \Omega_2}$ is
holomorphic 
in the first coordinate
and anti-holomorphic
in the second coordinate.
Thanks to the reproducing kernel property 
and Cauchy-Schwarz inequality, we obtain that
\[
\forall z \in \Omega_1, \forall w \in \Omega_2, 
\quad |\holkerK_\n(z,w)|^2 \leq \holkerK_\n(z,z)
\holkerK_\n(w,w).
\]
This allows us to apply Montel's Theorem,
and to show that any subsequence of the sequence 
$\holkerK_\n$ has at least one further subsequence
converging uniformly on compact subsets of 
$(\Omega_1\cup \Omega_2)\times (\Omega_1\cup \Omega_2)$.
On $\Omega_1 \times \Omega_1$ and $\Omega_2 \times \Omega_2$,
we have already established locally uniform convergence of the kernels.
The proof will be complete if we can show 
that for any $z \in \Omega_1$, we have
\[
\int_{\C \setminus \Omega_1} |\holkerK_\n(z,w)|^2 \diffA(w)
\xrightarrow[\n \to \infty]{} 0.
\]
To see why this convergence holds, let $S$ be a 
compact subset of $\Omega_1$. Then
\[
\holkerK_\n(z,z)=\int_{\C}|\holkerK_\n(z,w)|^2\diffA(w)
\ge \int_S|\holkerK_\n(z,w)|^2\diffA(w) +
\int_{\Omega_1^c}|\holkerK_\n(z,w)|^2\diffA(w).
\]
Since we know that $\holkerK_\n$ converges uniformly
towards $\BergK_1$ on the compact $S$, we get
\[
\BergK_{1}(z,z)\ge \int_S |\BergK_{1}(z,w)|^2\diffA (w)
+ \limsup_{\n\to\infty}\int_{\C \setminus \Omega_1}|
\holkerK_\n(z,w)|^2\diffA(w).
\]
Now, let $S$ exhaust $\Omega_1$, 
and notice that by the monotone convergence theorem we get
\[
\BergK_{1}(z,z)\ge \int_{\Omega_1} |\BergK_{1}(z,w)|^2\diffA(w)
+\limsup_{\n\to\infty}\int_{\C \setminus \Omega_1}
|\holkerK_\n(z,w)|^2\diffA(w)
\]
which, since $\BergK_1$ is a reproducing kernel on $\Omega_1$ yields
\[
\BergK_{1}(z,z) =\int_{\Omega_1} |\BergK_{1}(z,w)|^2\diffA(w),
\]
which completes the proof.

\section{Outliers for random polynomials}
This section is devoted to proving Theorem~\ref{thm:random-pol}.
We begin with a discussion of some preliminary material.

\subsection{Hardy spaces}
Denote by $\Omega$ a Jordan domain with 
analytic boundary (by convention, $\Omega$ is bounded), and let $\eta$
be a measure on $\partial\Omega$
with a continuous
and strictly positive
density with respect
to the arc length measure.
We define the Hardy space $H^2(\Omega,\eta)$ as the 
closure of analytic polynomials in the norm
\[
\lVert f\rVert_{H^2(\Omega,\eta)}^2:=\int_{\partial\Omega}|f|^2\,\diff\eta.
\] 
Analogously, if instead $\Omega$ is the complement of the closure of
a Jordan domain with analytic boundary, we define the Hardy space 
$H^2(\Omega,\eta)$ on $\C\setminus\overline{\Omega}$
as the closure of rational functions (with respect to the same norm)
with a fixed pole at a given
point $z_0\in\overline{\Omega}^c$. 

The Hardy space $H^2(\Omega,\eta)$ is a reproducing kernel Hilbert space, with
a Hermitian holomorphic reproducing kernel $S_{\Omega,\eta}(z,w)$
defined on $\Omega\times\Omega$, known as the \emph{Szeg\H{o}} kernel of 
$(\Omega,\eta)$.
If the measure is simply the arc length measure $\sigma_{\partial\Omega}$, 
we write $H^2(\Omega)$ and $S_\Omega$
for the space and the kernel, respectively. 
We denote by $H^2_\n(\Omega,\eta)$ the $\n+1$-dimensional 
polynomial Hardy space 
$\C_{\n}[z]\cap H^2(\Omega,\eta)$.

\begin{rem}\label{rem:multiplier-Hardy}
Let $\eta=\densityCurve \,\diff\sigma$, where 
$\densityCurve$ is strictly positive
and real-analytic.
Let $\fs=\fs_\densityCurve$ be a bounded holomorphic function on $\Omega$ 
whose real part equals $\frac12\log\densityCurve$
on the boundary $\partial\Omega$. Such a function is often referred to
as a \emph{Szeg\H{o}} function of $\densityCurve$. The function $\fs$ allows us to 
extend $\densityCurve$ to a log-harmonic function on $\Omega$
by putting $\densityCurve(z)=\big|\e^{2\fs (z)}\big|=\e^{2\Re \fs(z)}$. 
This implies that the multiplication operator $f\mapsto f\e^{-\fs}$ 
induces an isometric isomorphism between the spaces 
$H^2(\Omega)$ and $H^2(\Omega,\densityCurve\,\diff\sigma)$.
In particular the weighted Szeg\H{o} kernels satisfy the relation:
\[
\forall z,w \in \Omega, \quad S_{\Omega}(z,w) 
= S_{\Omega,\densityCurve} (z,w)  e^{ \fs(z) + \overline{ \fs(w) }}.
\]
In our main theorem on random polynomials (Theorem~\ref{thm:random-pol}), 
we work with the reproducing kernel
for polynomial subspaces of $L^2(e^{-2\n U^\nu}\diff\nu)$, 
where $\nu$ is supported on a regular Jordan curve. 
The above equality for the Szeg\H{o} kernel turns
out to be a crucial ingredient in our proof.
This type of equality would not be true for the corresponding Bergman kernels 
if the measure $\nu$ was supported on, e.g.,\ a domain. 
This fact is what prevents us from treating more general 
measures $\nu$ in the theorem.
\end{rem}

\subsection{Random polynomials}
Without loss of generality we may consider 
the case when $\Omega$ is bounded. 
The unbounded case
then follows by applying an appropriate inversion 
(see, for instance, \cite[Theorem~1.14]{BG}).

Notice that the Szeg\H{o} kernel $S_\Omega$ 
is the covariance kernel for the Gaussian analytic function
\[
P(z)=\sum_{k=0}^\infty \xi_k \psi_k(z),\qquad 
\xi_k\text{ i.i.d.},\; \xi_k\sim\mathcal{N}_{\C}(0,1),
\]
where $(\psi_k)_{k \geq 0}$
is any orthonormal basis
of $H^2(\Omega)$. 
Since by Remark~\ref{rem:multiplier-Hardy} 
$\psi_{k,\densityCurve}\overset{\rm def}{=}(\psi_k \e^{-\mathfrak{s}})_k$
is an orthonormal basis for $H^2(\Omega,\densityCurve\,\diff\sigma)$, 
it follows that the zero set of $P$ has the same
distribution as the zero set of the random analytic function 
$P_{\densityCurve}$, obtained similarly as a random linear combination
of the basis elements $\psi_{k,\densityCurve}$.

To prove Theorem~\ref{thm:random-pol} it is enough to show 
that the covariance kernel for the random polynomial \eqref{eq:rand-pol}, 
converges locally uniformly towards the 
Szeg\H{o} kernel $S_{\Omega,\densityCurve}$ 
(see, for instance, Section~2 in Shirai's paper \cite{Shirai}). 
Strictly speaking, this is not the case for our model, 
but it can be arranged by multiplying \eqref{eq:rand-pol} by an
appropriate zero-free holomorphic function.
Concretely, we introduce a holomorphic
function $\mathfrak{h}$ on $\Omega$ whose real part equals $U^\nu$
on $\partial\Omega$, and consider the random analytic function
\[
f_{\n}(z)=
\sum_{k=0}^{\n}
\xi_k Q_{k,\n}(z)
\e^{-\n \mathfrak{h}(z)} ,
\qquad \xi_k\text{ i.i.d.},\; \xi_k\sim\mathcal{N}_{\C}(0,1).
\]
Notice that $f_{\n}$ is a random linear 
combination of an orthonormal system in 
the Hardy space $H^2(\Omega,\densityCurve\,\diff\sigma)$, 
where we recall that $\nu=\densityCurve\,\diff\sigma$. 
The covariance kernel $\kerK_\n(z,w)$ of $f_N$ is given by
\[
\kerPol(z,w)=\sum_{k=0}^{\n}Q_{k,\n}(z)\overline{Q_{k,\n}(w)}
\e^{-\n\big(\mathfrak{h}(z)+\overline{\mathfrak{h}(w)}\big)}.
\]

\subsection{Proof of Theorem~\ref{thm:random-pol}} 
 \subsection*{Structure of the proof}
The proof is very similar to the proof of 
Theorem~\ref{thm:NonSimplyConnected}. We start by proving 
that $\kerPol(z,z)$ is bounded from above by the 
Szeg\H{o} kernel, in a similar spirit to 
Lemma~\ref{lem:monotone}. 
From this inequality, we can mimic 
Lemmas~\ref{lem: normal families} 
and \ref{lem:Herimitan} to reduce the local uniform 
convergence to the pointwise convergence of the kernels 
on the diagonal of $\Omega \times \Omega$. 
To obtain the pointwise convergence of the kernels, 
we construct, for fixed $k_0\in\N$, an orthonormal set
$(R_{k,\n})_{k \leq k_0}$ in $H_\n (\Omega, e^{-2\n U^{\nu}} \nu)$ 
such that, for any $k \leq k_0$ and any fixed $z \in \Omega$,
\begin{equation}\label{convergencebasispolynomials}
|R_{k,\n}(z)|^2 e^{-2\n U^{\mu}(z)} 
\xrightarrow[\n \to \infty]{} |\psi_k(z)|^2.
\end{equation}

Recall that we can express the covariance kernel $\kerPol$ 
by using any orthonormal basis of 
$H_\n (\Omega, e^{-2\n U^{\nu}} \densityCurve \diff \sigma)$. 
In particular we may chose a basis which extends the orthonormal 
set $(R_{k,\n})_{ k \leq k_0}$. Bounding from below by the sum of 
the first $k_0$ terms and applying \eqref{convergencebasispolynomials}, 
we obtain a lower bound depending on $k_0$. Choosing $k_0$ large enough, 
this lower bound comes arbitrarily close 
to the already established upper bound.

\subsection*{Normal families}
We observe that
\[
\kerPol(z,z)=\hspace{-2mm}
\sup_{p\in H_\n^2(\Omega,e^{-2\n U^\nu} \nu)}
\hspace{-3mm}
\frac{|p(z)|^2\e^{-2\n U^\nu(z)}}{\lVert p \rVert_\n^2} 
\le
\hspace{-2mm}
\sup_{f\in H^2(\Omega,\e^{-2\n U^\nu} \nu)}
\hspace{-3mm}
\frac{|f(z)|^2\e^{-2\n U^\nu(z)}}{\lVert f\rVert_\n^2},
\]
where, in this section, $\lVert \cdot\rVert_{\n}$ refers to the norm in
$L^2(\partial\Omega,\e^{-2\n U^\nu}\nu)$.
On the other hand, we may notice that,
since 
$f \in H^2(\Omega,\nu)$
if and only if
$g=f\e^{\n \mathfrak{h}}
\in H^2(\Omega, \e^{-2\n U^\nu}\nu)$, 
\begin{align*}
S_{\Omega,\rho}(z,z) 
& =
\sup_{f\in H^2(\Omega,\nu)}
\frac{|f(z)|^2}{
\phantom{a}\lVert f\rVert_{H^2(\Omega,\nu)}^2}
=
\sup_{g\in H^2(\Omega, \e^{-2\n U^\nu}\nu)}
\frac{\phantom{a}|g(z)\e^{-\n \mathfrak{h}(z)}|^2}
{\lVert g\rVert_{H^2(\Omega, e^{-2\n U^\nu}\nu)}^2} \\
& = \sup_{g\in H^2(\Omega,e^{-2\n U^\nu} \nu)}
\hspace{-3mm}
\frac{|g(z)|^2\e^{-2\n U^\nu(z)}}{\lVert g\rVert_\n^2},
\end{align*}
which implies that, for every $z \in \Omega$,
\begin{equation} \label{eq:monotonic Szego}
\forall \n\in\N, \forall z \in \Omega, \quad  \kerPol(z,z) 
\leq S_{\Omega,\rho}(z,z).
\end{equation}
From this equation, we can mimic Lemmas~\ref{lem: normal families} and 
\ref{lem:Herimitan} to reduce the local uniform convergence 
to pointwise convergence on the diagonal of $\Omega \times \Omega$.

\subsection*{Construction of the good orthonormal basis}
Let $(\psi_k)_k$ be an orthonormal basis of the space $H^2(\Omega, \nu)$, 
consisting of functions holomorphic on a slightly larger set $\Omega'$, 
containing $\overline{ \Omega}$. For instance, 
one may take $\psi_k=\sqrt{\phi'}\phi^{k}e^{-\fs}$, 
where $\phi$ is a fixed conformal map of $\Omega$ onto the unit disk $\D$, 
which extends conformally past the boundary of $\Omega$.
We define
the function $F_{k,\n}$ on $\mathbb C$ by
\[
F_{k,\n}=\chi\psi_k\e^{\n \mathfrak{h}}
\] 
where $\chi$ denotes a smooth non-negative cut-off
function which is identically one in a neighborhood of $\overline{\Omega}$ 
and vanishes identically on a neighborhood of $\C \setminus \Omega'$.
By construction, $F_{k,N}$ is a holomorphic
function on a neighborhood of $\overline{\Omega}$ 
and $(F_{k,N})_k$ is an orthonormal family in 
the weighted Hardy space $H^2(\Omega,e^{-2\n U^{\mu}}\diff \nu)$. 
Indeed, this follows from the identity
\[
F_{k,\n}\overline{F_{j,\n}} e^{-2\n U^{\nu}}\diff \nu
=\psi_k\overline{\psi_j}\diff\nu,
\]
which holds since $\chi$ is identically one on $\Gamma$
and since, by construction, 
\[
\e^{-2\n U^\nu} = \big| \e^{-N \mathfrak h} \big|^2.\]
Now, choose a $C^{1,1}$-smooth subharmonic function 
$V:\mathbb C \to \mathbb R$, 
harmonic on a neighbourhood of $\overline \Omega$, 
that coincides with $U^\nu$ on $\overline \Omega$,
satisfies 
$\limsup_{z \to \infty}
 \{V(z) - \log|z|\} < \infty$ 
and such that 
there is a  $\delta>0$
with
\[
\Delta V(z) \geq \delta 
\mbox{ and } V(z) \geq \Re \mathfrak{ h}(z) 
\text{ whenever } \nabla \chi (z) \neq 0.\]
This function can be constructed as in the
proof of
Lemma~\ref{lem:changeofV} 
(taking $\varphi$ from
\eqref{eq:VasIntegralPhi} to be zero
on a neighborhood of $0$).
By Lemma~\ref{lem:horm} applied with $\phi =
2\n V$ and
\[
\bar \partial
F_{k,\n} = 
(\bar \partial \chi)\psi_k\e^{\n \mathfrak{h}},
\]
there exists a function $v_{k,\n}$ such that
$\bar \partial  v_{k,\n}=\bar \partial F_{k,\n}$ and
\begin{align}
\label{eq:horm-pf-eq12}
\int_\C \big|v_{k,\n}\big|^2 \e^{-2\n V}
\diff \m 
& \le \int_{\supp \nabla \chi}
\frac{\big|\bar \partial F_{k,\n}\big|^2}{2 \n \Delta V}
\e^{-2\n V}\diff\m \\
&\le
\int_{\supp \nabla \chi}
\frac{\big|\bar \partial \chi \big|^2}{2 \n \delta}
\big|\psi_{k} \big|^2
\e^{-2\n \left(V- \Re \mathfrak h \right)}
\diff\m \xrightarrow[\n \to \infty]{} 0.
\end{align}
In particular, due to the
harmonicity of $V$ 
and the holomorphicity of
$v_{k,\n}$ on a neighborhood
of $\overline{\Omega}$ we have also that
$v_{k,\n} \e^{-N V}\xrightarrow[\phantom{aa}]{} 0$
uniformly on $\overline{\Omega}$. Then,
\[R_{k,\n} = F_{k,\n} -  v_{k,\n} \mbox{ is holomorphic
on } \mathbb C \quad \mbox{ and } \quad
\int_\C \big|R_{k,\n}\big|^2 \e^{-2\n V}
\diff \m < \infty\] which implies that
$R_{k,\n}$ is a polynomial
of degree at most $\n-2$.
Moreover, 
\begin{equation}
\label{eq:PointwiseConvergenceR}
R_{k,\n}e^{-\n \mathfrak h} 
\xrightarrow[\n \to \infty]{} \psi_k
\end{equation}
uniformly on $\overline{\Omega}$.
In particular, we have that
\begin{equation}
\label{eq:OrthogonalityR}
\int_{\partial\Omega}
R_{k,\n} \overline{R_{m,\n}} e^{-2\n U^\nu} \diff \nu
\xrightarrow[\n \to \infty]{} 
\int_{\partial\Omega}
\psi_k \overline \psi_m \diff \nu
=
\delta_{k,m}.
\end{equation}
Fix $k_0>0$ and let $(P_{k,\n})_{k=0}^{k_0}$
be the Gram-Schmidt orthonormalization 
of $(R_{k,\n})_{k=0}^{k_0}$.
This can be done
since, by \eqref{eq:OrthogonalityR},
 $(P_{k,\n})_{k=0}^{k_0}$
is a linearly independent family
for $\n$ is large enough.
By extending
$(P_{k,\n})_{k=0}^{k_0}$ to
an orthonormal
basis of $H^2_\n(\Omega, e^{-2\n U^\nu} \diff \nu)$
we see that
\[\sum_{k=0}^{k_0}
|P_{k,\n}(z)|^2 e^{-2\n U^\nu(z)} \leq 
\kerPol(z,z).
 \]
By \eqref{eq:PointwiseConvergenceR} and
\eqref{eq:OrthogonalityR},
$P_{k,\n}e^{-\n \mathfrak h} (z)
\xrightarrow[\phantom{aa}]{} \psi_k(z)$
as $\n \to \infty$
for every $z \in \Omega$ so that
\[\sum_{k=0}^{k_0} |\psi_k(z)|^2 
\leq \liminf_{\n \to \infty}
\kerPol(z,z) \]
for every $z \in \Omega$
and, by letting $k_0 \to \infty$,
we complete the proof.

\subsection*{Independence} The proof of the independence assertion in the
theorem is identical to the proof of Theorem~\ref{thm:indep}, because
the Szeg\H{o} kernels also have the reproducing property. 
In actuality this proves that the Gaussian fields are independent in
the limit, which then implies that the zero processes are also independent.
We omit the details.

\appendix
\section{Determinantal Point Processes}
\label{app:David}

In this section we recall some basic properties of
determinantal point processes that are needed in this paper.
More information
may be found in the book
\cite{HoughKrisPeresVirag} of Hough, Krishnapur, Peres and Vir\'{a}g or 
in Soshnikov's survey \cite{Soshnikov}.

Let $U$ be an open subset of $\mathbb C$
and denote by $\mathcal C_U$ the set of
locally finite integer-valued positive
measures $\mu$ on $U$.
Any element $\mathcal X$ of
$\mathcal C_U$ can we written as
$\mathcal X = \sum_{\lambda \in \Lambda}
\delta_{x_\lambda}$
for some locally finite family
$\left(x_\lambda\right)_{\lambda \in \Lambda}$ of points
of $U$. This allows us
to think of $\mathcal X \in \mathcal
C_U$ as a multi-set, and we use notations such as
$\sum_{x \in \mathcal X} f(x)=
\int f \mathrm d\mathcal X
$ or
$\# (A\cap \mathcal X)
= \mathcal X(A)$.
Endow $\mathcal C_U$ with the coarsest topology
that makes the maps
\[\mathcal X \in \mathcal C_U
\mapsto
\sum_{x \in \mathcal X} f(x)\]
continuous, for every
compactly supported
continuous function
$f: U \to \mathbb R$.
By a \emph{point process}
on $U$
we mean a Borel probability measure on 
$\mathcal C_U$.

\begin{defn}[Determinantal point process]
Suppose that $\mu$ is
an atomless locally finite
measure on $U$  and let
$\ker:U \times U \to \mathbb C$
be continuous and Hermitian, i.e.,
$\ker(z,w) = \overline{K(w,z)}$
for every $z,w \in U$.
A point process $\mathcal X$ on $U$
is the \emph{determinantal point process}
associated with
the kernel $K$ with respect
to the measure $\mu$
if  for every
$\ker \geq 1$
and for every
pairwise disjoint
measurable sets
$A_1,\dots,A_k$,
\[
\mathbb E
\left[ \# (A_1
\cap \mathcal X) \dots
\# (A_k \cap \mathcal X) \right]
=
\int_{A_1 \times \dots \times A_k}
\hspace{-6mm}
\det\big( \ker(x_i,x_j)_{1\le i,j\le k}\big)
\mathrm d \mu^{\otimes_k}
(x_1,\dots,x_k),
\]
where $\mu^{\otimes_k}$ denotes the $k$-fold
product measure of $\mu$ with itself. 
\end{defn}

There is an important
freedom in the choice of
$\ker$ and $\mu$ that we use
throughout the article. Namely,
if $f: U \to \mathbb C\setminus \{0\}$
is continuous and if we let
\begin{equation}
\label{eq:ChangeOfK}
\widetilde \ker(z,w) =
f(z) \ker(z,w) \overline{f(w)}
\quad \text{ and }
\diff \widetilde \mu
=|f|^{-2}\mathrm d\mu
\end{equation}
then the determinantal point
process associated
to $K$ with respect to $\mu$ is also
is also determinantal with kernel  $\widetilde \ker$
with respect to the measure
$\widetilde {\mu}$.

A well-known example
of a determinantal point process
is the following.
We recall that $\C_{\n-1}[z]$
denotes the set of polynomials
of degree at most
$\n-1$.

\begin{prop}[Coulomb gases are determinantal at $\beta=2$]
\label{prop:CoulombAreDeterminantal}
Let $\mu$
be an atomless
measure
on $\mathbb C$ such that for some $N\ge 2$,
$\int (1+|x|^{2})^{N-1}
\diff \mu(x) < \infty$. Then,
the measure
$\gamma$ on $\mathbb C^n$
defined by
\[
\diff \gamma(x_1,\dots,x_\n)
=\prod_{i<j}^\n
|x_i - x_j|^2
\diff \mu^{\otimes_\n}
(x_1,\dots,x_\n)
\]
is a finite measure, and we obtain a probability measure
by writing 
$\mathbb{P}_\n=\frac{\gamma}{\gamma(\C^\n)}$.
Then a random element $(X_1,\dots,X_\n)$ that
 follows the law 
$\mathbb{P}_\n$ induces
a determinantal point process
$\{X_1,\dots,X_\n\}$
with respect to the reference measure $\mu$
and associated with the kernel 
\[
\ker_\n(z,w) =
\sum_{k=1}^\n
p_k(z) \overline{p_k(w)},
\]
where
$\{p_1,\dots,p_{\n}\}$
is an orthonormal basis
of $L^2(\mathbb C,\mu)\cap\C_{\n-1}[z]$.
\end{prop}

\begin{proof}
We first remark that the determinant
$\det \big(\ker_\n(x_j,x_k)_{1\le j,k \le \n})$
is proportional to $\prod_{i<j} |x_i-x_j|^2$ so
the measure $\gamma$ is proportional to 
$\det \big(\ker_\n(x_j,x_k)_{1\le j,k \le \n})\diff\mu^{\otimes_\n}$.
Since $\ker_\n$ is an orthogonal projection onto
a space of dimension $\n$ we have the reproducing property
\[
\int \ker_\n(x,y) \ker_\n(y,z) \diff \mu (y) = \ker_\n(x,z)
\]
and $\int \ker_\n(x,x) \diff \mu(x) = \n$. This implies 
(see Lemma~5.27 in Deift's book \cite{Deiftbook} for details) that
\[
\int  \det(\ker_\n(x_i,x_j)_{1\le i,j\le k})
\diff \mu(x_k) = (\n-k+1)
\det\big( \ker_\n(x_i,x_j)_{1\le i,j\le k-1}\big),
\]
from which the theorem follows by iterated integration.
\end{proof}

Usually, we use
Proposition~\ref{prop:CoulombAreDeterminantal}
with
$\diff \mu =
\e^{-2V}
\diff \m_{\mathbb C}$
for some
continuous function
$V: \C \to \R$.
In that case,
by  \eqref{eq:ChangeOfK}
we may choose
$\m_{\C}$
as the reference measure
for the determinantal point
process $\mathcal X$ and
$\ker(z,w)\e^{-(V(z)+V(w))}$
as the kernel.

The next result is Proposition 3.10 in the work of
Shirai and Takahashi
\cite{ShiraiTakahashi}.

\begin{prop}[Uniform convergence of kernels]
\label{prop:UCVofKernels}
Let $U$ be an open
set of $\mathbb C$ and, for
every positive integer $\n$, let
$\ker_\n:U \times U \to \mathbb C$
be a continuous
Hermitian function.
Fix an atomless
locally finite measure $\mu$
on $\C$
and, for each $\n$, suppose
$\mathcal X_\n$ is a determinantal
point process associated with
$\ker_\n$ with respect to $\mu$. Then, if
$\ker_\n$ converges uniformly on compact sets to
$\ker:U \times U \to \C$,
there exists a determinantal
point process $\mathcal X$
associated with $\ker$ with respect to
$\mu$ such that
\[
\mathcal X_\n
\xrightarrow[\n \to \infty]
{\text{law}} \mathcal X.
\]
\end{prop}

In the next proposition
we have disjoint
open subsets
$U_1$ and $U_2$ of $\C$, kernels
$\ker_1: U_1 \times U_1
\to \mathbb C$ and
$\ker_2: U_2 \times U_2
\to \mathbb C$ and measures
$\mu_1$ on $U_1$ and
$\mu_2$ on $U_2$.
Both kernels may be
extended by zero to the rest of
$(U_1 \cup U_2) \times
(U_1 \cup U_2)$ so that
it makes sense to write
$\ker_1 + \ker_2$. Similarly,
we may extend both measures by zero
to the rest of $U_1 \cup U_2$ so that
it makes sense to write $\mu_1+\mu_2$.

\begin{prop} \label{prop:UnionOfDPP}
Let $U_1$ and $U_2$ be disjoint open
subsets of $\mathbb C$.
Suppose that for $j=1,2$, $\mathcal X_j$
is a determinantal point process on $U_j$
associated with $K_j$
with respect to $\mu_j$.
Then, if $\mathcal{X}_1$ and $\mathcal{X}_2$
are independent, the union $\mathcal X_1 \cup \mathcal X_2$
is a determinantal point process
on $U_1 \cup U_2$
associated with
$\ker=\ker_1 + \ker_2$ with respect to
$\mu=\mu_1 + \mu_2$.
\end{prop}

\begin{proof}Recall that, with the above convention, we have $\ker=\ker_j$
on $U_j\times U_j$.
Fix $k \geq 1$ and let $A_1, \dots, A_k$ be $k$ measurable
subsets of $U_1 \cup U_2$.
Introduce the sets $A_i^{(1)} = A_i \cap U_1$
and $A_i^{(2)} = A_i \cap U_2$
so that
\[
A_i=A_i^{(1)} \uplus A_i^{(2)}
\quad \text{ with } \quad 
A_i^{(1)} \subset U_1 
\text{ and } A_i^{(2)} \subset U_2. 
\]
We want to show that
\[
\mathbb E
\left[ \# (A_1
\cap \mathcal X) \dots
\# (A_k \cap \mathcal X) \right]
= 
\int_{A_1 \times \dots \times A_k}
\hspace{-6mm}
\det\big(\ker(x_i,x_j)_{1\le i,j\le k}\big)
\mathrm d \mu^{\otimes_k}
(x_1,\dots,x_k).
\]
By linearity of the expectation, we have that
\[
\mathbb E
\left[ \# (A_1
\cap \mathcal X)\dots
\# (A_k \cap \mathcal X) \right]
=\hspace{-3mm}
\sum_{\sigma_1,\dots,\sigma_k
\in \{1,2\}}\mathbb E
\big[ \# (A_1^{(\sigma_1)}
\cap \mathcal X_{\sigma_1}) \dots
\# (A_k^{(\sigma_k)} \cap \mathcal X_{\sigma_k}) \big].
\]
On the other hand,
if $F_k(x_1,\dots,x_k)=
\det\big( K(x_i,x_j)_{i,j 
\in \{1,\dots,k\}}\big)$,
we also have
\[
\int_{A_1 \times \dots \times A_k}
F_k \mathrm d \mu^{\otimes_k}
= \sum_{\sigma_1,\dots,\sigma_k
\in \{1,2\}}
\int_{A_1^{(\sigma_1)} \dots \times A_k^{(\sigma_k)}}
F_k \mathrm d( \mu_{\sigma_1}
\otimes \dots \otimes 
\mu_{\sigma_k}).
\]
Then, we only need to show that
\[
\mathbb E
\big[ \# (A_1^{(\sigma_1)}
\cap \mathcal X) \dots
\# (A_k^{(\sigma_k)} \cap \mathcal X) \big]
=\int_{A_1^{(\sigma_1)} 
\dots \times A_k^{(\sigma_k)}}
F_k \mathrm d( \mu_{\sigma_1}
\otimes \dots \otimes 
\mu_{\sigma_k}).
\]
Since $F_k$ is symmetric, we may suppose that
there is $m \in \{0,\dots,k\}$ such that
\[
\sigma_1=\dots =\sigma_m = 1\quad
\text{and}\quad \sigma_{m+1} = \dots = \sigma_k = 2.
\]
By the independence of $\mathcal X_1$
and  $\mathcal X_2$, it holds that
\begin{align*}
\mathbb E&
\big[ \# (A_1^{(\sigma_1)}
\cap \mathcal X_{\sigma_1}) \dots
\# (A_k^{(\sigma_k)} \cap \mathcal X_{\sigma_k}) \big]
\\&=\mathbb E
\left[ \# (A_1^{(1)}
\cap \mathcal X_{1}) \dots
\# (A_m^{(1)}
\cap \mathcal X_{1}) \right]
\mathbb E \left[
\# (A_{m+1}^{(2)}
\cap \mathcal X_{2})  \dots
\# (A_k^{(2)} \cap \mathcal X_{2}) \right]			
\\
&=\int_{A_1^{(1)}
\times \dots \times A_m^{(1)}}
F_m \mathrm d\mu_{1}^{\otimes_m}
\int_{A_{m+1}^{(2)}
\times \dots \times A_{k}^{(2)}}
F_{k-m} \mathrm d\mu_{2}^{\otimes_{k-m}}.
\end{align*}
Moreover, by the calculation
of the determinant of a diagonal block matrix,
\[F_m(x_1,\dots,x_m)
F_{k-m}(x_{m+1},\dots,x_k)
=F_k(x_1,\dots,x_k)\]
whenever $x_1,\dots,x_m \in U_1$
and $x_{m+1},\dots,x_k \in U_2$
which, by Fubini's theorem,
allows us to conclude the proof.
\end{proof}

\section{Harmonic extension of potentials from uncharged regions}
\label{app:proof-harm-ext}
In this section we prove Proposition~\ref{prop:extension1}
i.e. that under the admissibility assumptions on $\mu$ of Definition~\ref{def:admissibility},
the harmonic function $U^\mu\vert_{\Omega}$ extends
harmonically past $\partial\Omega$ to a 
function $\calU=\calU^\mu$ on an open set
$\Omega'$, containing $\overline{\Omega}$.

\begin{proof}[Proof of Proposition~\ref{prop:extension1}]
First of all, notice that
the statement
is local. Indeed,
suppose we show that
for every $z \in \partial \Omega$
there exists a small open disk
$D$ centered at
$z$ and a harmonic function
on $D$ that coincides with
$U^\mu$ on $\Omega \cap D$.
Then, we may cover $\partial \Omega$ by a family of open sets
$\{D_\lambda\}_{\lambda \in \Lambda}$ 
that satisfy that intersections are connected, that
if $D_{\lambda_1} \cap D_{\lambda_2} \neq \emptyset$ then
$D_{\lambda_1} \cap D_{\lambda_2} \cap \Omega \neq \emptyset$, 
and such that for each $\lambda \in \Lambda$ there exists
a harmonic function
$f_\lambda:D_\lambda \to \mathbb R$ satisfying
$f_\lambda|_{\Omega \cap D_\lambda} = U^\mu|_{\Omega \cap D_\lambda}$.
We may define a function $f$ on $\mathcal N := \cup_\lambda D_\lambda$
by $f(z) = f_\lambda(z)$ if $z \in U_\lambda$. Notice that if
$z \in D_{\lambda_1} \cap D_{\lambda_2}$ then
$D_{\lambda_1} \cap D_{\lambda_2} \neq \emptyset$ so that
$D_{\lambda_1} \cap D_{\lambda_2} \cap \Omega \neq \emptyset$, i.e.,
$f_{\lambda_1} = f_{\lambda_2}$ on a nonempty open set which implies that
$f_{\lambda_1} = f_{\lambda_2}$ on $D_{\lambda_1} \cap D_{\lambda_2}$
since $D_{\lambda_1} \cap D_{\lambda_2}$ is connected.

Now fix $z \in \partial \Omega$
and 
choose a parametrization
of $\partial \Omega$ near $z$, i.e.,
a holomorphic function
$\varphi: 
(-1,1)
\times (-1,1)
\to \mathbb C$
such that 
$\varphi|_{(-1,1)
\times\{0\}}$
is a usual parametrization
of $\partial \Omega$
around $z$ with
$\varphi(0) = z$
and $\varphi'(0) \neq 0$. 
By taking
a restriction of $\varphi$
if necessary
and reparameterizing,
we have that
 $\varphi\big((-1,1)
\times (-1,0]\big)
\subset \mathbb C \setminus 
\overline{\Omega} $
and
$\varphi\big((-1,1)
\times (0,1)\big)
\subset \Omega $.
Let us use the notation
 $A=\varphi\big((-1,1)
\times (-1,0]\big)$
and write
\[U^\mu = U^{\mu|_A} + 
U^{\mu|_{A^c}}.\]
Notice that
$U^{\mu|_{A^c}}$
is already harmonic on 
$\varphi\big((-1,1)
\times (-1,1)\big)$
so that we only need to
extend harmonically
$U^{\mu|_A}$ 
 to $A$.

Now, the new setting is the following.
We define the square
$\mathcal S = (-1, 1)
\times (-1, 1)$
and $\mu = 
\varphi^* (\mu|_A)$
which is a measure on (the redefined)
$A = (-1,1)\times(-1,0]$.
We have a function 
$U: \mathcal S
\to \mathbb R$
such that $\Delta U = 
2\pi\mu$
and we wish to find a harmonic extension of
$U|_{(-1,1) 
\times (0,1)}$
to the open square
$\mathcal S$. The first 
fact to notice is that
being able to extend $U$
is independent of $U$
as soon as 
$\Delta  U = 2\pi \mu$.
Indeed, if $V$
were another function
such that $\Delta V = 2\pi\mu$
and if $\mathcal U$
were a harmonic extension
of $U|_{(-1,1) 
\times (0,1)}$,
the function
$\mathcal U
- U +  V$
would be a harmonic extension of 
$ V|_{(-1,1) 
\times (0,1)}$.
So, we have the freedom
to choose $U$
and we will use the nice choice
\[U (z)=
U^{\mu}(z)
= \int \log|z-w|
\mathrm d\mu(w).\]
Recall that 
$\mathrm d\mu  = \densityBulk  \mathrm d m_{\mathcal S} +
\densityCurve \mathrm d m_{(-1,1)}$
for some functions
$ \densityBulk: (-1,1)\times (-1,0]  \to \mathbb R$ and
$\densityCurve: (-1,1) \to \mathbb R$ which are real-analytic.
Denote 
$\densityBulk_y(x) =  \densityBulk(x,y)$
and, for any $y \in (-1,0]$, define the measure 
$\mathrm d\mu_y =  \densityBulk_y \mathrm d m_{(-1,1) \times \{y\}} $.
Denote by $\eta$ the
measure on $(-1,1)$ such that
 $\mathrm d \eta = \densityCurve \mathrm dm_{(-1,1)}$. Then,
since $\mu = 
\int_{-1}^0 
\mu_y \mathrm d y + \eta$, we have
\[U(z)
= 
\int_{-1}^0
\left(\int \log|z-w| \mathrm d \mu_s(w)
\right) \mathrm d s
+ 
\int \log|z-w| \mathrm d \eta(w).
\]
We assume, without loss of generality,
that $\densityBulk(x,y)= \sum_{m,n \geq 0} a_{m,n} x^m y^n$  
for every $(x,y) \in (-1,1) \times (-1,0]$
and we extend it to  $\densityBulk: \mathcal S \times (-1,1) \to \mathbb C$ 
by assuming, without loss of generality, that
$\densityBulk(z,y) =  \sum_{m,n \geq 0} a_{m,n} (z - iy)^m y^n$
is convergent.
Moreover, we can assume that the series
defining $\densityBulk$ converges
in $(2\mathcal S) \times (-1,1)$
and that $\densityCurve$ can be extended to a holomorphic function
on $2 \mathcal S$.
Fix $y \in (-1,0]$ and take a path $\gamma^{(y)}$ in the lower half-plane 
that connects 
$(-1,y)$ and $(1,y)$ avoiding
$\mathcal S$.  
Define $\mathcal U$ by
\[\mathcal U(z)
= \int_{-1}^0
\left(\mathrm{Re}
\int_{\gamma^{(y)}} \log(z-w) \densityBulk_y(w) \mathrm d w
\right) \mathrm d y
+
\mathrm{Re}\int_{\gamma^{(0)}}
\log(z-w) \densityCurve(w) \mathrm d w,\]
where 
$\log:\mathbb C \setminus i \mathbb R_- \to \mathbb C$
is any branch of 
the logarithm. Note that, for $w\in \gamma^{(y)}$ and
$z\in \mathcal{S}$,  $z-w$ belongs to $\C\setminus i\R_{-}$.
Since $\densityBulk_y$ and $\densityCurve$ are
holomorphic on $\mathcal S$, 
$U(z) = \mathcal U(z)$ for $z\in \mathcal S \setminus A$
and, since $\mathcal U$ is harmonic
on $\mathcal S$, the proof is complete.
\end{proof}

\subsection*{Acknowledgements}
We want to thank Mikhail Sodin, Haakan Hedenmalm and Ofer Zeitouni 
for stimulating discussions. We also want to 
thank Mikhail Sodin and Ofer Zeitouni for
supporting a visit of G.-Z.\ to Israel.

The figures were created using IPE.

\subsection*{Financial support}
R.\ B. was supported by the ERC Advanced grant No.\ 692452 
and by the MSC Grant No.\ 898769 (RandPol).
D.\ G.-Z. is supported by the grant ANR-16-CE40-0024.
A.\ N. was supported by ISF Grant 1903/18 and 
BSF Start up Grant No.\ 2018341.
A.\ W. was supported by ERC Advanced Grant No.\ 692616,
KAW Foundation Grant No.\ 2017.038, and by Grant No.\ 2022-03611
from the Swedish Research Council (VR).

\end{document}